\documentclass[11pt]{article}
\usepackage{amsthm,amsmath,mathrsfs,enumitem,amssymb,mathtools,mathrsfs}
\usepackage{xcolor}
\usepackage{hyperref}
\usepackage[english]{babel}
\usepackage{graphicx}
\usepackage[hypcap]{caption}
\usepackage{enumitem}
\usepackage[normalem]{ulem}

\theoremstyle{plain}
\newtheorem{theorem}{Theorem}[section]
\newtheorem{lemma}[theorem]{Lemma}
\newtheorem{proposition}[theorem]{Proposition}
\newtheorem{corollary}[theorem]{Corollary}

\theoremstyle{definition}
\newtheorem{definition}[theorem]{Definition}
\theoremstyle{remark}
\newtheorem{remark}[theorem]{Remark}
\numberwithin{equation}{section}

\setlist{nosep}
\newcommand{\indf}{1_}
\newcommand*\samethanks[1][\value{footnote}]{\footnotemark[#1]}

\title{Large Time Behaviour and the Second Eigenvalue Problem for Finite State Mean-Field Interacting Particle Systems}
\author{Sarath Yasodharan\thanks{Supported by the Indo-French Centre for Applied Mathematics.} \thanks{Supported by a fellowship grant from the Centre for Networked Intelligence (a Cisco CSR initiative) of the Indian Institute of Science, Bengaluru.} and Rajesh Sundaresan\samethanks[1]
\\ Indian Institute of Science}

\begin{document}
\maketitle 
\begin{abstract}
This article examines large time behaviour of finite state mean-field interacting particle systems. Our first main result is a sharp estimate (in the exponential scale) on the time required for convergence of the empirical measure process of the $N$-particle system to its invariant measure; we show that when time is of the order of $\exp\{N\Lambda\}$ for a suitable constant $\Lambda \geq 0$, the process has mixed well and it is close to its invariant measure. We then obtain large-$N$ asymptotics of the second largest eigenvalue of the generator associated with the empirical measure process when it is reversible with respect to its invariant measure. We show that its absolute value scales as $\exp\{-N\Lambda\}$. The main tools used in establishing our results are the large deviation properties of the empirical measure process from its large-$N$ limit. As an application of the study of large time behaviour, we also show convergence of the empirical measure of the system of particles to a global minimum of a certain `entropy' function when particles are added over time in a controlled fashion. The controlled addition of particles is analogous to the cooling schedule associated with the search for a global minimum of a function using the simulated annealing algorithm.
\vspace{10pt}

\end{abstract}

\section{Introduction}
\label{section:introduction}
In this paper, we study large time behaviour and the second eigenvalue problem for Markovian mean-field interacting particle systems with jumps. Our motivation is to provide an understanding of  metastable phenomena in engineered systems such as load balancing networks~\cite{aghajani-etal-17,aghajani-ramanan-19,mukhopadhyay-etal-16,mitzhenmaker-00,graham-00}, wireless local area networks~\cite{bhattacharya-kumar-17,benaim-leboudec-08,bordenave-etal-12,kumar-etal-06,ramaiyan-etal-08,bianchi-98}, and in natural systems involving grammar acquisition, sexual evolution~\cite{panageas-vishnoi-16,panageas-etal-16}, epidemic spread~\cite{leonard-90,djehiche-kaj-95}, etc.  These systems are briefly described in Section~\ref{section:examples}.

Before we discuss our main contributions, let us describe the setting of our mean-field interacting particle system.
\subsection{The setting}
Let there be $N$ particles. Each particle has a state associated with it which comes from a finite set $\mathcal{Z}$; the state of the $n$th particle at time $t$ is denoted by $X_n^N(t) \in \mathcal{Z}$. The empirical measure of the system of particles at time $t$ is defined by
\begin{align*}
\mu_N(t) \coloneqq \frac{1}{N} \sum_{n=1}^N \delta_{X_n^N(t)} \in M_1(\mathcal{Z}),
\end{align*}
where $\delta_{\cdot}$ denotes the Dirac measure  on $\mathcal{Z}$. Here, $M_1(\mathcal{Z})$ denotes the space of probability measures on $\mathcal{Z}$ equipped with a metric that generates the topology of weak convergence\footnote{Since $\mathcal{Z}$ is a finite set, the total variation metric on $M_1(\mathcal{Z})$ generates this topology.} on $M_1(\mathcal{Z})$. Each particle has a set of allowed transitions; to define this, let $(\mathcal{Z}, \mathcal{E})$ be a directed graph with the interpretation that whenever $(z,z^\prime) \in \mathcal{E}$, a particle in state $z$ is allowed to move from $z$ to $z^\prime$.  To specify the interaction among the particles and  the evolution of the states of the particles over time, for each $(z,z^\prime) \in \mathcal{E}$, we are given a function $\lambda_{z,z^\prime}: M_1(\mathcal{Z}) \to [0, \infty)$. We consider the generator $\Psi^N$ acting on functions $f$ on $\mathcal{Z}^N$ by

\begin{align*}
\Psi^N f(\mathbf{z}^N) = \sum_{n=1}^N \sum_{z_n^\prime: (z_n, z_n^\prime) \in \mathcal{E}} \lambda_{z_n, z_n^\prime}(\overline{\mathbf{z}^N}) (f(\mathbf{z}^N_{n,z_n,z_n^\prime}) - f(\mathbf{z}^N));
\end{align*}
here $\overline{\mathbf{z}^N} = \frac{1}{N}\sum_{n=1}^N \delta_{z_n} \in M_1(\mathcal{Z})$ denotes the empirical measure associated with the configuration $\mathbf{z}^N \in \mathcal{Z}^N$, and $\mathbf{z}^N_{n,z_n,z_n^\prime}$ denotes the resultant configuration of the particles when the $n$th particle changes its state from $z_n$ to $z_n^\prime$.

We make the following assumptions on the model:
\begin{enumerate}[label=({A\arabic*})]
\item The graph $(\mathcal{Z}, \mathcal{E})$ is irreducible. \label{assm:a1}
\item The functions $\lambda_{z,z^\prime}(\cdot)$, $(z,z^\prime) \in \mathcal{E}$, are Lipschitz continuous on $M_1(\mathcal{Z} )$ and there exist positive constants $c, C $ such that $c \leq \lambda_{z,z^\prime}(\xi) \leq C$ for all $(z,z^\prime) \in \mathcal{E}$ and all $\xi \in M_1(\mathcal{Z})$. \label{assm:a2}
\end{enumerate}

Let $D([0, \infty), \mathcal{Z}^N)$ denote the space of $\mathcal{Z}^N$-valued functions on $[0, \infty)$ that are right continuous with left limits (c\`adl\`ag), equipped with the Skorohod-$J_1$ topology (see~{\cite[Chapter~3]{ethier-kurtz}}). Since the transition rates are bounded (by assumption~\ref{assm:a2}), the $D([0, \infty), \mathcal{Z}^N)$-valued martingale problem for $\Psi^N$ is well posed (see~{\cite[Exercise~15,~Section~4.1]{ethier-kurtz}}); therefore, given an initial configuration of the particles $(X_n^N(0), 1 \leq n \leq N) \in \mathcal{Z}^N$, we have a Markov process $\left( (X_n^N(t), 1 \leq n \leq N),  t \geq 0\right)$ whose sample paths are elements of $D([0,\infty),\mathcal{Z}^N)$. To describe the process in words, a particle in state $z$ at time $t$ moves to state $z^\prime$ at rate $\lambda_{z,z^\prime}(\mu_N(t))$ independent of everything else; i.e., the evolution of the state of a particle depends on the states of the other particles via the empirical measure of the states of all the particles, hence the name mean-field interaction. Note that the empirical measure process $(\mu_N(t), t \geq 0)$ is also a Markov process with state space  $M_1^N(\mathcal{Z})$ which is the  set of elements of $M_1(\mathcal{Z})$ that can arise as empirical measures of $N$-particle configurations on $\mathcal{Z}^N$. Its generator $L^N$ acting on functions $f$ on $M_1^N(\mathcal{Z})$ is given by
\begin{align*}
L^N f(\xi) = N \sum_{(z,z^\prime) \in \mathcal{E}} \xi(z) \lambda_{z,z^\prime}(\xi) \left[f\left( \xi + \frac{\delta_{z^\prime}}{N} - \frac{\delta_z}{N}\right) - f(\xi)  \right].
\end{align*}
Since $\mu_N$ is a Markov process on a finite state space, and since the graph $(\mathcal{Z},\mathcal{E})$ of allowed particle transitions is irreducible (Assumption~\ref{assm:a1}), there exists a unique invariant probability measure for $\mu_N$, which we denote by $\wp_N$. Also, let $P_\nu$ denote the law of $(\mu_N(t), t \geq 0)$ with initial condition $\mu_N(0) = \nu \in M_1^N(\mathcal{Z})$ (i.e.~the solution to the $D([0,\infty),M_1(\mathcal{Z}))$-valued martingale problem for $L^N$ with initial condition $\nu \in M_1^N(\mathcal{Z})$)  and let $E_\nu$ denote integration with respect to $P_\nu$; in both $P_\nu$ and $E_\nu$ we suppress the dependence on $N$ for ease of readability.

\newpage
\subsection{Main results}
Let us now discuss the main results of the paper.
\subsubsection{Convergence to the invariant measure}
\label{subsection:main-result-1}
Our first main result is on the time required for the process $\mu_N$ to equilibrate. This time grows at an exponential rate with the number of particles $N$ where the rate is the constant $\Lambda$ which will be defined in~\eqref{eqn:lambda-def}. 
\begin{theorem} Given $\delta > 0$ there exist $\varepsilon > 0$ and $N_0 \geq 1 $ such that, with $T = \exp\{N(\Lambda+\delta)\}$,
\begin{align*}
\sup_{\nu \in M_1^N(\mathcal{Z})} \left| E_\nu (f(\mu_N(T) ))- \langle f, \wp_N \rangle \right| \leq \|f\|_\infty \exp\{-\exp(N\varepsilon)\}
\end{align*}
for all $N \geq N_0$ and all bounded Borel-measurable functions $f$ on $M_1(\mathcal{Z})$.
\label{thm:conv}
\end{theorem}
The result says that when time is of the order $\exp\{N(\Lambda+\delta)\}$ for any $\delta > 0$, the process has mixed well and it is close to its invariant measure. The proof of this result is based on the study of large time behaviour of the process $\mu_N$. Before we describe this, let us mention a well-known law of large numbers for the process $\mu_N$~\cite{mckean-67,gartner-88,sznitman-91,benaim-leboudec-08}. This will not only pave the way for a suitable description of the constant $\Lambda$ but also lead us to a converse of Theorem~\ref{thm:conv} and the significance of $\Lambda$.

Assume~\ref{assm:a1} and \ref{assm:a2}, and suppose that the initial conditions $\{\mu_N(0)\}_{N \geq 1}$ converge weakly to a deterministic measure $\nu \in M_1(\mathcal{Z})$. Then for any fixed $T > 0$, the empirical measure process $(\mu_N(t), 0 \leq t \leq T)$ converges in $D([0,T],M_1(\mathcal{Z}))$, in probability, to the solution to the ODE
\begin{align}
\dot{\mu}(t) = \Lambda_{\mu(t)}^* \mu(t), \, 0 \leq t \leq T, \, \mu(0) = \nu,
\label{eqn:MVE}
\end{align}
where, for any $\xi \in M_1(\mathcal{Z})$, $\Lambda_{\xi}$ denotes the $|\mathcal{Z}|\times |\mathcal{Z}|$ rate matrix\footnote{The rate matrix is given by $\Lambda_\xi(z,z^\prime)= \lambda_{z,z^\prime}(\xi)$ when $(z,z^\prime) \in \mathcal{E}$, $\Lambda_{\xi}(z,z^\prime) = 0$ when $(z,z^\prime) \notin \mathcal{E}$, and $\Lambda_{\xi}(z,z) = -\sum_{z^\prime \neq z} \lambda_{z,z^\prime}(\xi)$ for all $z \in \mathcal{Z}$.} when the  empirical measure is $\xi$, $\Lambda^*_\xi$ denotes its transpose, and $D([0,T],M_1(\mathcal{Z}))$ denotes the space of $M_1(\mathcal{Z})$-valued c\`adl\`ag functions on $[0,T]$ equipped with the Skorohod-$J_1$ topology (we assume that all paths are left continuous at $T$). The above ODE is referred to as the~McKean-Vlasov equation. The above convergence result enables one to view the process $\mu_N$ as a small random perturbation of the ODE~(\ref{eqn:MVE}).

We now elaborate on the large time behaviour of $\mu_N$. Suppose that the limiting McKean-Vlasov equation~(\ref{eqn:MVE}) has multiple $\omega$-limit sets (multiple stable equilibria and/or limit cycles). If we focus on a fixed time interval $[0,T]$, let the number of particles $N \to \infty$, and let the initial conditions $\mu_N(0)$ converge weakly to a deterministic limit $\nu$, then the mean-field convergence  suggests that the empirical measure process tracks the solution to the McKean-Vlasov equation~(\ref{eqn:MVE}) over $[0,T]$ starting at $\nu$. If we then let $T \to \infty$, the solution to the McKean-Vlasov equation goes to an $\omega$-limit set of~(\ref{eqn:MVE}) depending on the initial condition $\nu$. On the other hand, for a large but fixed $N$, the process would track the McKean-Vlasov equation with high probability and, as time becomes large, would thus enter a neighbourhood of the $\omega$-limit set corresponding to the initial condition $\nu$; however, because of the randomness in the finite-$N$ system, the process can exit the basin of attraction of this $\omega$-limit set. It is then likely to remain in a neighbourhood of another $\omega$-limit set for a large amount of time before transiting to the next one, and so on. These are examples of  \emph{metastable phenomena,} and it turns out that the sojourn times in the basin of attraction of an $\omega$-limit set are of the order $\exp\{O(N)\}$, as we shall soon see. The proof of Theorem~\ref{thm:conv} exploits quantitative estimates of the following metastable phenomena,
\begin{enumerate}[label=({\roman*})]
\item the mean time spent by the process near an $\omega$-limit set,
\item the probability of first reaching a particular $\omega$-limit set's neighbourhood before reaching  the neighbourhood of another one, and
\item  the probability of traversing the neighbourhoods of a given set of $\omega$-limit sets in a particular order.
\end{enumerate}
These quantifications are important in their own right as they help predict the performance of engineered systems, some of which we will describe in Section~\ref{section:examples}. We study the aforementioned metastability questions in Section~\ref{section:large_time_behaviour}. Such large time phenomena for diffusion processes with a small noise parameter have been studied in the past by Freidlin and Wentzell~\cite{freidlin-wentzell} under the ``general position condition" (see {\cite[Sections~6.4-6.6]{freidlin-wentzell}}). Hwang and Sheu~\cite{hwang-sheu-90} studied large time behaviour for diffusion processes under a more general setup. The key in both these works is the large deviation properties of the small noise diffusion processes over finite time durations, which have been established in {\cite[Chapter~5]{freidlin-wentzell}}. In this paper, we extend the analysis to Markov mean-field jump processes, specifically $(\mu_N(\cdot))_{N \geq 1}$.

The proof of Theorem~\ref{thm:conv} is carried out using lower bounds (Theorem~\ref{thm:mixing}) for the probability that, starting from any point in $M_1^N(\mathcal{Z})$, the process $\mu_N$ is in a small neighbourhood of one of the \emph{most stable}\footnote{See Section~\ref{subsection:convergence-invariant-measure} for a precise definition.} $\omega$-limit set(s) of the McKean-Vlasov equation~(\ref{eqn:MVE}) when time is of the order $\exp\{N(\Lambda-\delta_0)\}$, for a small $\delta_0>0$. The constant $\Lambda$ is defined using  ``costs of passages" between the $\omega$-limit sets of the McKean-Vlasov equation~\eqref{eqn:MVE}. These costs are quantified in terms of the large deviations rate function associated with the process $\mu_N$ via certain graphs called $W$-graphs (see Section~\ref{subsection:L-subsets} for the definition of $W$-graphs). See~\eqref{eqn:lambda-def} for a precise definition of $\Lambda$.

Our next result is, in a certain sense, a converse of Theorem~\ref{thm:conv}. Let $i_0$ be one of the most stable $\omega$-limit set(s) of~\eqref{eqn:MVE}.
\begin{theorem}
There exist $\nu_0 \in M_1(\mathcal{Z})$,  $\delta > 0$, $\beta >0$, $\rho_1 > 0$ and $N_0 \geq 1$ such that, with $T = \exp\{N(\Lambda-\delta)\}$, 
\begin{align*}
P_{\nu}(\mu_N(T) \in (\text{the } \rho_1 \text{ neighbourhood of } i_0)) \leq \exp\{-N\beta\}.
\end{align*}
for all $\nu$ in the $\rho_1$-neighbourhood of $\nu_0$ in $M_1^N(\mathcal{Z})$ and $N \geq N_0$.
\label{thm:conv-converse}
\end{theorem}
In other words, when time is of the order $\exp\{N(\Lambda - \delta) \}$, there are initial conditions $\nu \in M_1^N(\mathcal{Z})$ such that  the probability that $\mu_N(\exp\{N(\Lambda - \delta) \})$ is in a small neighbourhood of one of the most stable $\omega$-limit set(s) is exponentially small. The process is then not likely to have equilibrated because it has not visited a set with high invariant measure. Thus, Theorem~\ref{thm:conv} and Theorem~\ref{thm:conv-converse} together indicate that the constant $\Lambda$ is sharp (in the exponential scale) for the time required for equilibration of $\mu_N(\cdot)$.

A convergence result similar to that of Theorem~\ref{thm:conv} for the mean-field discrete-time setting but without the specification of the constant $\Lambda$ was established by Panageas and Vishnoi~\cite{panageas-vishnoi-16}. Let us reemphasise that our setting is a continuous-time setting. To identify the constant $\Lambda$ in  this setting, we must study the large deviation asymptotics in greater detail. Theorems~\ref{thm:conv} and~\ref{thm:conv-converse} combine time and the number of particles. Additionally, Theorem~\ref{thm:conv} is a statement that holds uniformly over all initial conditions unlike convergence bounds (over time) for a fixed number of particles with a given initial condition, e.g.~\cite{thai-15}. The proof of Theorem~\ref{thm:conv} is inspired by that of Hwang and Sheu's {\cite[Theorem~2.1,~Part~I]{hwang-sheu-90}} where similar results are established for small noise diffusions.

\subsubsection{Asymptotics of the second largest eigenvalue}
Our second main result is on the asymptotics of the second largest eigenvalue of the generator $L^N$ of the Markov process $\mu_N=(\mu_N(t), t \geq 0)$ when it is reversible with respect to its invariant measure $\wp_N$. For a fixed $N$, the convergence speed of the process $\mu_N$ to its invariant measure (over time) can be understood by studying the modulus of the second largest eigenvalue of $L^N$. We show that the modulus of the second largest eigenvalue of $L^N$ (which we denote by $\lambda_2^N$) scales as $\exp\{-N\Lambda\}$; here $\Lambda$~(defined in \eqref{eqn:lambda-def}) is the constant that appears in the statement of Theorem~\ref{thm:conv}. More precisely,
\begin{theorem}
\begin{align*}
\lim_{N \to \infty} \frac{1}{N} \log \lambda_2^N = -\Lambda.
\end{align*}
\label{thm:eval_problem}
\end{theorem}
It turns out that $\Lambda$ can be positive only when there are metastable states in the limiting dynamics~(\ref{eqn:MVE}) (i.e.~when~(\ref{eqn:MVE}) possesses multiple $\omega$-limit sets). In such situations, one expects slower convergence to the invariant measure for large values of $N$. On the other hand, $\Lambda$ can be $0$, for example, when the limiting dynamics~(\ref{eqn:MVE}) has a unique globally asymptotically stable equilibrium; in this special case,  convergence of $\mu_N$ to its invariant measure does not suffer from the slowing down phenomenon associated with positive $\Lambda$. In fact, Panageas and Vishnoi~\cite{panageas-vishnoi-16} and Panageas et al.~\cite{panageas-etal-16} show that the mixing time is $O(\log N)$ in the discrete-time setting. Kifer~\cite{kifer-90} considers a more restrictive discrete-time model, which does not cover the mean-field model, and identifies the constant analogous to $\Lambda$~\cite[Theorem~4.3]{kifer-90}. The restriction is that the state space of $\mu_N$ is the same for each $N$ and that a certain uniform finite duration large deviation principle should hold with the rate function satisfying a continuity property. One can view our result as an extension of Kifer's~\cite[Theorem~4.3]{kifer-90} to the continuous-time \emph{mean-field} setting, where the state space of the Markov process $\mu_N$ changes with $N$. Hwang and Sheu~\cite{hwang-sheu-90} establish a result similar to ours on the scaling of the second largest eigenvalue of a reversible small noise diffusion process, and our method of proof is inspired by their approach.
\subsubsection{Convergence to a global minimum via controlled addition of particles}
Our third main result is on the convergence of the empirical measure process to a global minimum of a natural `entropy' function when particles are injected over time at a specific rate reminiscent of the simulated annealing algorithm's cooling schedule, $N(t) = \lfloor \frac{\log (2+t)}{c^* + \delta}\rfloor$ for a suitable $c^*$ and any $\delta > 0$. This entropy function is the large deviations rate function associated with the sequence of invariant measures $\{\wp_N, N \geq 1\}$, which is in turn defined in terms of the large deviations rate function associated with the process $\mu_N$; see~\eqref{eqn:s_expression} for its definition. 

Fix $c>0$. Let $N_0 = \min\{n \in \mathbb{N} : \exp\{nc\} - 2 \geq 0 \}$, $t_{N_0} = 0$, and for each $N > N_0$, let $t_N = \exp\{Nc\}-2$. We construct a process with controlled addition of particles as follows. We start with $N_0$ particles with certain initial states and let the process evolve according to the generator $L^{N_0}$ until time $t_{N_0+1}$. For each $N > N_0$, we add an extra particle at time $t_N$, and for a fixed state $z_0 \in \mathcal{Z}$, we set the state of the new particle to $z_0$ and let the process evolve according to the generator $L^{N}$ from $t_{N}$ to $t_{N+1}$ (see a more precise description of the process in Section~\ref{section:conv_global_minimum}). Let $\bar{\mu}$ denote the above time-inhomogeneous Markov process and let $P_{0,\nu}$ denote the law of $\bar{\mu}$ on $D([0,\infty), M_1(\mathcal{Z}))$ with initial condition $\bar{\mu}(0) = \nu$. Also, let $\tilde{L}_0$ denote the set of all global minima of the entropy function (see Section~\ref{subsection:thm-conv-proof} for the precise definition of $\tilde{L}_0$). Our convergence result is the following.
\begin{theorem}
There exists a constant $c^*>0$ such that for all $c > c^*$ and any $\rho_1 > 0$,
\begin{align*}
P_{0,\nu} (\bar{\mu}(t) \in (\text{the } \rho_1\text{-neighbourhood of } \tilde{L}_0)) \to 1
\end{align*}
as $t \to \infty$, uniformly for all $\nu \in M_1^{N_0}(\mathcal{Z})$.
\label{thm:conv-globalmin}
\end{theorem}
Note that the convergence to a global minimum holds for all starting points. This is of use in situations where a population growth schedule is applied in order to {\em engineer} the mean-field system's movement to a desired equilibrium point, as time $t \to \infty$. One can also use this approach to study numerically the most likely region in which the process $\mu_N$ spends time for large values of $N$, under stationarity. Again, our proof is inspired by the analysis of the simulated annealing algorithm in {\cite[Part III]{hwang-sheu-90}}.

\subsection{Key ingredients for the proofs}
\label{subsection:proofs-note}
The proofs of our main results follow the outlines in \cite{hwang-sheu-90}. However, in order to make them work in our present context (which involves jump Markov processes and the mean-field setting), we need to establish the following properties:
\begin{itemize}
\item a uniform version of the finite-duration large deviation principle for $\{(\mu_N(t), 0 \leq t \leq T), N \geq 1\}$, where the uniformity is over the initial condition;
\item continuity of the cost function associated with movement between points on the simplex $M_1(\mathcal{Z})$;
\item strong Markov property of $\mu_N(\cdot)$.
\end{itemize}
The key insight from this paper is the abstraction of these three properties and their importance in establishing the large time behaviour and metastability properties of mean-field systems. We leverage the results of \cite{borkar-sundaresan-12} to establish the above properties.

We now describe the key ideas in each of the main results.

To prove Theorem~\ref{thm:conv}, one possible approach is to wait long enough for the process $\mu_N$ to hit a neighbourhood of one of the most stable $\omega$-limit set(s) of~(\ref{eqn:MVE}), regardless of the initial condition, and then allow sufficient additional time for the process to mix well. We prove Theorem~\ref{thm:conv} using this idea; we first consider a sequence of passages of $\mu_N$ between neighbourhoods of $\omega$-limit sets of~(\ref{eqn:MVE}) to reach one of the most stable $\omega$-limit set. Each of these passages take place between ``stable" subsets of $\omega$-limit sets called cycles (see Section~\ref{subsection:cycles}). Probability of each of these passages over time intervals of the form $\exp\{N\times \text{constant}\}$ for appropriate constants can be lower bounded, thanks to the uniform large deviation property of $\mu_N$ (see~Theorem~\ref{thm:mixing}). We then tie them up using the strong Markov property of $\mu_N$. These steps yield a lower bound on the transition probability for $\mu_N$ (see~Corollary~\ref{cor:lowerboundtransition}) and Theorem~\ref{thm:conv} follows as a consequence of this. We can also produce an upper bound for probability of these passages for suitable initial conditions if enough time has not lapsed (see~\eqref{eqn:ub_tpm} in Theorem~\ref{thm:mixing}). Theorem~\ref{thm:conv-converse} follows as a consequence of this upper bound.

Theorem~\ref{thm:eval_problem} follows from an application of Theorem~\ref{thm:conv}. We use the spectral expansion of the generator of $\mu_N$, when it is reversible with respect to its invariant measure $\wp_N$, and the large deviation principle for $\{\wp_N, N \geq 1\}$ to prove Theorem~\ref{thm:eval_problem}.

In Theorem~\ref{thm:conv-globalmin}, to bring the process $\mu_N$ to one of the most stable $\omega$-limit set(s) of~(\ref{eqn:MVE}) (i.e., one of the global minima of our entropy function), regardless of the initial condition, we introduce new particles over time in a controlled fashion. Before reaching a global minimum, the system may possibly explore other local minima. Since addition of particles amounts to reduction of ``noise" in the process $\mu_N$, we must make sure that particles are introduced sufficiently slowly over time so that the system does not get trapped in a local minimum. This is achieved by the choice of our particle addition schedule $N(t), t \geq 0$, which is the analogue of the cooling schedule in simulated annealing. The schedule also enables us to apply the uniform large deviation principle over sufficiently long time durations to $\bar{\mu}$ so as to extend the results on large time behaviour used in the proof of Theorem \ref{thm:conv} to the present situation when the number of particles change over time (see Lemma~\ref{lemma:hs22}-\ref{lemma:hs24}). These extensions along with the method to analyse the passages of the system through cycles, the idea used in the proof of Theorem~\ref{thm:conv}, enables us to prove a $1- o(1)$ lower bound on the probability that $\bar{\mu}(t)$ belongs to a neighbourhood of a global minimum of our entropy function as $t \to \infty$, no matter where we start the process.

\subsection{Examples}
\label{section:examples}
The mean-field interacting particle system that we have described can be used to model many interesting phenomena that arise in various domains such as physics, engineering, biology, etc. In this section, we shall describe some applications that are relevant to communication networks and shall point to the related literature that study these applications via mean-field models. Naturally, the examples and the related literature that we have mentioned below are by no means exhaustive.

The first example is load balancing in networks. We describe the simplest model, the power of two choices, studied by Mitzenmacker~\cite{mitzhenmaker-00}. Here, each particle is a single server $M/M/1$ queue, and the state represents the number of customers waiting in the queue. In load balancing, one is interested in routing the incoming customers to an appropriate queue so as to minimise the average delay experienced by a customer. The obvious way to do this is to route the customer to a queue with the least number of waiting customers. But, since there are a large number of queues, polling all of them and finding the ones with the least number of customers is expensive. So a simple alternative is to pick a queue at random and route the incoming customer to that queue, which is studied in~\cite{graham-00}. It turns out that, if we pick two queues at random and route the customer to the least loaded queue between the two (with ties broken uniformly at random), the delay decreases dramatically. This algorithm demonstrates the power of two choices, and the evolution of the state of each queue under this algorithm can be described using the mean-field model which has been used to analyse the delay performance~\cite{mitzhenmaker-00}. For  related problems on load balancing in networks, see Mukhopadhyay~et~al.~\cite{mukhopadhyay-etal-16} who study heterogeneous servers, Aghajani et~al.~\cite{aghajani-etal-17,aghajani-ramanan-19} who study non-Markovian queues, etc., and the references therein. Note that one important difference with our setting is that the state space of a queue is countably infinite in this class of problems. The finite state space model arises in the above settings when the buffers are finite and packets arriving at a fully buffered queue are lost.

Another example arises in the modelling of a wireless local area network (WLAN). Here, each particle is a wireless node trying to access a common medium, and the state of a particle represents the aggressiveness with which a packet transmission is attempted. The nodes interact with each other via the medium access control (MAC) protocol implemented in the system. Whenever a wireless node encounters a collision due to a transmission from another node, it changes its state to a less aggressive one, and whenever it succeeds, it changes its state to a more aggressive one. Therefore, the evolution of the state of a node depends on the empirical measure of the states of all the nodes, as in our mean-field model. This model was first proposed by Bianchi~\cite{bianchi-98} and has proved to be useful in analysing the performance of the MAC protocol; other works that focus on the WLAN application include: Bordenave et al.~\cite{bordenave-etal-12} who studied a two time scale mean-field interacting particle system with a fast varying background process to model partial interference among nodes, Kumar et~al.~\cite{kumar-etal-06} who used the mean-field model to study the performance of WLANs using a fixed-point analysis,  Ramaiyan et~al.~\cite{ramaiyan-etal-08} and Bhattacharya and Kumar~\cite{bhattacharya-kumar-17} who looked at the problem of short term unfairness using the aforementioned fixed-point analysis, etc. Note that our model is a continuous-time modification of the discrete-time models in the above papers. Yet the continuous-time model provides accurate predictions on the discrete-time model; see~{\cite[page~4]{borkar-sundaresan-12}}. Some papers work directly with the continuous-time model; see, for example,~Boorstyn~et~al.~\cite{boorstyn-etal-87}.

Other applications that use the mean-field model include analysis and control of spread of epidemics in networks \cite{benaim-leboudec-08,akhil-etal-19,leonard-90,djehiche-kaj-95}, dynamic routing in circuit-switched networks~\cite{anantharam-91}, scheduling in cellular systems~\cite{manjerkar-etal-14}, game-theoretic modelling and analysis of behaviour of agents in societal networks \cite{reiffers-sundaresan,li-etal-15},  etc.
\subsection{Outline of the paper}
The rest of the paper is organised as follows. In Section~\ref{section:ldp_finite}, we discuss large deviation principles for the empirical measure process $\mu_N$ over a finite time horizon. These play an important role in the study of large time behaviour of $\mu_N$ and the large deviation principle for the invariant measure $\{\wp_N\}_{N\geq 1}$. We then study the large time behaviour of the process $\mu_N$ in Section~\ref{section:large_time_behaviour}, and prove our first main result on the proximity of the law of $\mu_N$ to its invariant measure. In Section~\ref{section:eval_problem}, we study the asymptotics of the second largest eigenvalue of the generator of the process $\mu_N$ in the reversible case. Finally, in Section~\ref{section:conv_global_minimum}, we study the convergence of the empirical measure process to a global minimum of the aforementioned entropy function when particles are injected into the system at a suitable rate.

\section{Preliminaries: Large deviations over finite time durations}
\label{section:ldp_finite}
In this section, we present a large deviation principle for the process $\mu_N$ over finite time durations. This result will be used later to study the large-time behaviour of $\mu_N$ and the rate of convergence of $\mu_N$ to its invariant measure.

Fix $T>0$. We introduce some notations. Let $p_{\nu_N}^{(N)}$ denote the solution to the  $D([0,T], M_1(\mathcal{Z}))$-valued martingale problem for $L^N$, i.e., the law of the empirical measure process $(\mu_N(t), 0 \leq t \leq T)$, and let $p_{\nu_N,T}^{(N)}$ denote the law of the terminal-time empirical measure $\mu_N(T) \in M_1(\mathcal{Z})$, with a deterministic initial condition $ \mu_N(0) = \nu_N $. Let $\mathcal{AC}[0,T]$ denote the space of absolutely continuous $M_1(\mathcal{Z})$-valued paths on $[0,T]$ (in particular they are differentiable for almost all $t \in [0,T]$; see~{\cite[Definition~3.1]{leonard-95}}). Define
\begin{align*}
\tau^*(u) \coloneqq \left\{
\begin{array}{lll}
\infty & \text{ if } u < -1 \\
1 & \text{ if } u = -1 \\
(u+1) \log (u+1) - u& \text{ if } u > -1,
\end{array}
\right.
\end{align*}
which is the Fenchel-Legendre transform of $\tau(u) = e^u-u-1, u \in \mathbb{R}$. Recall the definition of the family of rate matrices $(\Lambda_\xi, \xi \in M_1(\mathcal{Z}))$ from Section~\ref{section:introduction}.
We have the following large deviation principle (LDP) for the sequence $\{p_{\nu_N}^{(N)}\}_{N \geq 1}$ on $D([0,T], M_1(\mathcal{Z}))$ (see~{\cite[Theorem~3.1]{leonard-95}},~{\cite[Theorem~3.2]{borkar-sundaresan-12}}). See~{\cite[Section~1.2]{dembo-zeitouni}} for the definition of LDP and a good rate function.
\begin{theorem}
Suppose that the initial conditions $\nu_N \to \nu$ in $M_1(\mathcal{Z})$. Then the sequence of probability measures $\{p_{\nu_N}^{(N)}, N\geq 1\}$ on the space $D([0,T], M_1(\mathcal{Z}))$ satisfies the LDP with good rate function $S_{[0,T]}(\cdot |\nu)$ defined as follows. If $\mu(0) = \nu$ and $\mu \in \mathcal{AC}[0,T]$, then
\begin{align*}
S_{[0,T]}(\mu|\nu)  & = \int_{[0,T]} \sup_{\alpha \in \mathbb{R}^{|\mathcal{Z}|}} \biggr\{ \sum_{z \in \mathcal{Z}} \alpha(z)(\dot{\mu}_t(z) - \Lambda_{\mu_t}^*\mu_t(z)) \\
& \qquad  - \sum_{(z,z^\prime) \in \mathcal{E}} \tau(\alpha(z^\prime) - \alpha(z)) \lambda_{z,z^\prime}(\mu_t) \mu_t(z) \biggr\} dt,
\end{align*}
and $S_{[0,T]}(\mu | \nu) = +\infty$ otherwise. Moreover, if $S_{[0,T]}(\mu|\nu) < \infty$, then there exists a unique family of rate matrices $L(t) = (l_{z,z^\prime}(t), z,z^\prime \in \mathcal{Z}), 0 \leq t \leq T$, such that $t \mapsto L(t)$ is measurable, $\mu$ is the solution to
\begin{align*}
\dot{\mu}(t) = L(t)^* \mu(t), \, 0 \leq t \leq T, \, \mu(0) = \nu,
\end{align*}
and
\begin{align*}
& S_{[0,T]}(\mu|\nu) =  \int_{[0,T]} \sum_{(z,z^\prime)\in \mathcal{E}} \mu(t)(z) \lambda_{z,z^\prime}(\mu(t)) \tau^*\left( \frac{l_{z,z^\prime}(t)}{\lambda_{z,z^\prime}(\mu(t))} -1 \right) dt,
\end{align*}
where $L(t)^*$ denotes the transpose of $L(t)$, $t \in [0,T]$.
\label{thm:finite-duration-ldp}
\end{theorem}
We can interpret the rate function $S_{[0,T]}$ as follows. Starting at $\nu_N$, the process $\mu_N$ is likely to be in the neighbourhood of the solution to the McKean-Vlasov equation~(\ref{eqn:MVE}) with initial condition $\nu$ (with high probability). In order for the process $\mu_N$ to be in the neighbourhood of some other path, we need to apply a control given by the rate matrix $L$; $S_{[0,T]}(\mu|\nu)$ is the cost of this control. In particular, since the solution to the McKean-Vlasov equation starting at $\nu$ has zero-cost (i.e. $S_{[0,T]}(\mu_{\nu}|\nu) = 0$ where $\mu_\nu$ denotes the solution to~(\ref{eqn:MVE}) starting at $\nu$), the limiting behaviour that $\mu_N(\cdot) \xrightarrow{P} \mu_\nu(\cdot)$ in $D([0,T],M_1(\mathcal{Z}))$ as $N \to  \infty$ follows.

Here is an outline of the proof of Theorem~\ref{thm:finite-duration-ldp}: one looks at a system of non-interacting particles where the transition rates of a particle do not depend on the  empirical measure, and considers the corresponding  empirical measure process over $[0,T]$. Since at most one particle can jump at a given point of time, the measure $p_{\nu_N}^{(N)}$ is absolutely continuous with the measure corresponding to the above non-interacting system on $D([0,T],M_1(\mathcal{Z}))$. One can then write the Radon-Nikodym derivative using the Girsanov formula and show continuity properties of the same. An application of an extension of Sanov's theorem (see~{\cite[Theorem~3.5]{dawson-gartner-87}}) tells us that the non-interacting particle system obeys the LDP on $D([0,T],M_1(\mathcal{Z}))$. The above theorem then follows by an application of Varadhan's integral lemma (see~{\cite[Theorem~4.3.1]{dembo-zeitouni}}). This approach has been carried out for a system of interacting diffusions in~\cite{dawson-gartner-87} and for jump processes in~\cite{leonard-95,borkar-sundaresan-12}. One can also prove various special cases of Theorem~\ref{thm:finite-duration-ldp} via other simpler methods; for example, for fixed initial conditions, i.e., when $\nu_N = \delta_{z}$ for some $z \in \mathcal{Z}$ and for all $N \geq 1$, one can use a modification of Varadhan's lemma to obtain the LDP for $p_{\delta_{z}}^{(N)}$ (see~\cite{delmoral-zajic-03}), but letting the initial condition to be arbitrary, except for the constraint $\nu_N \to \nu$ weakly, is crucial to obtain a uniform version of the Theorem~\ref{thm:finite-duration-ldp} (see Corollary~\ref{cor:uniform_ldp}), which is used prove our main results.

We now recall a theorem that gives the large deviation principle for the sequence $\{p_{\nu_N,T}^{(N)}\}_{N \geq 1}$ on $M_1(\mathcal{Z})$. This can be obtained from the above theorem by an application of the contraction principle to the coordinate projection map $D([0,T], M_1(\mathcal{Z})) \ni \mu \mapsto \mu(T)$ (see~\cite[Theorem~4.2.1]{dembo-zeitouni}, \cite[Theorem~3.3]{borkar-sundaresan-12}).
\begin{theorem}
Suppose that the initial conditions $\nu_N \to \nu$ in $M_1(\mathcal{Z})$.  Then the sequence of probability measures $\{p_{\nu_N,T}^{(N)}\}_{N \geq 1}$ on the space $M_1(\mathcal{Z})$ satisfies the LDP with the good rate function
\begin{align*}
S_T(\xi|\nu) \coloneqq \inf \{S_{[0,T]}(\mu|\nu): &\,  \mu(0) = \nu, \mu(T) = \xi,  \mu \in \mathcal{AC}[0,T]\}.
\end{align*}
Moreover, the above infimum is attained, i.e., there exists a path $\hat{\mu} \in \mathcal{AC}[0,T] $ such that $\hat{\mu} (0)=\nu,\,  \hat{\mu} (T) = \xi$ and $S_{[0,T]}(\hat{\mu}|\nu) = S_T(\xi|\nu)$.
\label{thm:ldp_terminal_time}
\end{theorem}

Here, $S_T(\xi|\nu)$ can be interpreted as the minimum cost of passage from the profile $\nu$ to the profile $\xi$ in time $T$, among all paths from $\nu$ to $\xi$ in time $T$. It can be shown that $S_T$ is continuous on $M_1(\mathcal{Z}) \times M_1(\mathcal{Z})$ by constructing piecewise constant velocity trajectories between points on $M_1(\mathcal{Z})$ (see~\cite[Lemma~3.3]{borkar-sundaresan-12}).

We also have the following \emph{uniform} LDP for the sequence $\{p_{\nu_N}^{(N)}\}_{ N \geq 1}$ (see~{\cite[Corollary~3.1]{borkar-sundaresan-12}}) when the initial condition is allowed to lie in a compact set.
\begin{corollary}
For any compact set $K \subset M_1(\mathcal{Z})$, any closed set $F \subset D([0,T], M_1(\mathcal{Z}))$, and any open set $G \subset D([0,T], M_1(\mathcal{Z}))$, we have
\begin{align}
\limsup_{N \to \infty} \frac{1}{N} \log \sup_{\nu \in K \cap M_1^N(\mathcal{Z})} p_\nu^{(N)} \{\mu_N \in F\} \leq - \inf_{\nu \in K } \inf_{\mu \in F} S_{[0,T]} (\mu|\nu),
\end{align}
and
\begin{align}
\liminf_{N \to \infty} \frac{1}{N} \log \inf_{\nu \in K \cap M_1^N(\mathcal{Z})} p_\nu^{(N)} \{\mu_N \in G\} \geq - \sup_{\nu \in K} \inf_{\mu \in G} S_{[0,T]} (\mu|\nu).
\end{align}
\label{cor:uniform_ldp}
\end{corollary}
For a proof of the above, see~{\cite[Corollary~5.6.15]{dembo-zeitouni}}. Note that, since the space $M_1(\mathcal{Z})$ is compact, we may take $K = M_1(\mathcal{Z})$ in the above corollary.

\begin{remark}The version of uniform LDP presented in Corollary~\ref{cor:uniform_ldp} is slightly different from the definition of uniform LDP in Freidlin and Wentzell~\cite[Section~3, Chapter~3]{freidlin-wentzell}. The version presented here suffices for proofs our main results since our state space $M_1(\mathcal{Z})$ is compact and the rate function $S_T$ defined in Theorem~\ref{thm:ldp_terminal_time} is continuous (see~\cite[Theorem~2.7]{salins-19} and~\cite[Appendix~A]{borkar-sundaresan-12}).
\end{remark}
\section{Large time behaviour}
\label{section:large_time_behaviour}
In the study of large-time behaviour of $\mu_N$, an important role is played by the Freidlin-Wentzell quasipotential  $V : M_1(\mathcal{Z}) \times M_1(\mathcal{Z}) \to [0, \infty)$ defined by
\begin{align*}
V(\nu,\xi) \coloneqq \inf \{S_{[0,T]}(\mu|\nu):\mu(T)= \xi, T > 0\},
\end{align*}
i.e., $V(\nu, \xi)$ denotes the minimum cost of transport from $\nu$ to $\xi$ in an arbitrary but finite time.

We say that $\nu \sim \xi$ ($\nu$ is equivalent to $\xi$) if $V(\nu, \xi) = 0$ and $V(\xi, \nu) = 0$. It is easy to see that $\sim$ defines an equivalence relation on $M_1(\mathcal{Z})$. To study the large time behaviour of the process $\mu_N$, we make the following assumptions on the McKean-Vlasov equation~(\ref{eqn:MVE}) (see~\cite[Chapter~6,~Section~2,~Condition~A]{freidlin-wentzell}):
\begin{enumerate}[label=({B\arabic*})]
\item There exists a finite number of compact sets $K_1, K_2, \ldots, K_l$ such that
\begin{itemize}
\item For each $i = 1,2, \ldots l$,  $\nu_1, \nu_2 \in K_i$ implies $\nu_1 \sim \nu_2$.
\item For each $i \neq j$, $\nu_1 \in K_i$ and $\nu_2 \in K_j$ implies $\nu_1 \nsim \nu_2$.
\item Every $\omega$-limit set of the dynamical system~(\ref{eqn:MVE}) lies completely in one of the compact sets $K_i$.
\end{itemize}
\label{assm:b1}
\end{enumerate}
Since $V(\nu_1, \nu_2) = 0$ whenever $\nu_1,\nu_2 \in K_i$ for any $1 \leq i \leq l$, we can define
\begin{align*}
V(K_i, K_j) \coloneqq \inf\{S_{[0,T]}(\mu|\nu):\nu \in K_i, \mu(T) \in K_j, T > 0\},
\end{align*}
which is interpreted as the minimum cost of going from $K_i$ to $K_j$. We also define the minimum cost of going from $K_i$ to $K_j$ without touching the other compact sets $K_k, k \neq i,j$ by
\begin{align*}
\tilde{V}(K_i, K_j) & \coloneqq \inf  \{S_{[0,T]}(\mu|\nu):\nu \in K_i, \mu(t) \notin \cup_{k\neq i, j} K_k  \\
& \qquad \text{ for all } 0 \leq t \leq T, \mu(T) \in K_j, T>0\}.
\end{align*}

\subsection{Preliminary results}
It turns out that, under assumption~\ref{assm:b1}, the large time behaviour of the process $\mu_N$ can be studied via a discrete time Markov chain whose state space is the union of small neighbourhoods of the compact sets $K_i, 1 \leq i \leq l$. To study this chain, we introduce some notation. Let $L = \{1,2,\ldots, l\}$. Given $0 < \rho_1 < \rho_0$, let $\gamma_i$ (resp.~$\Gamma_i$) denote the $\rho_1$-open neighbourhood (resp.~$\rho_0$-open neighbourhood) of $K_i$.  Let $\gamma = \cup_{i=1}^l \gamma_i$,  $\Gamma = \cup_{i=1}^l \Gamma_i$, and $C =M_1(\mathcal{Z}) \setminus \overline{\Gamma}$. For a set $A \subset M_1(\mathcal{Z})$ and $\delta>0$, let $[A]_\delta$ denote the $\delta$-open neighbourhood of $A$, and for a subset $W \subset L$, abusing notation, let $[W]_\delta$ denote the $\delta$-open neighbourhood of $\cup_{i \in W} K_i$. For each $n \geq 1$, we define the sequence of stopping times: $\tau_0\coloneqq 0$, $\sigma_n \coloneqq \inf\{t>\tau_{n-1}: \mu_N(t) \in C\}$, $\tau_n \coloneqq \inf\{t > \sigma_n : \mu_N(t) \in \gamma\}$, and define $Z^N_n \coloneqq \mu_N(\tau_n)$.  Since $\mu_N$ is strong Markov, $Z^N$ is a discrete time Markov chain, and $Z_n^N \in \gamma \cap M^N_1(\mathcal{Z})$ for all $n \geq 1$. For a measurable set $A \in M_1(\mathcal{Z})$, we define the stopping time $\tau_A \coloneqq \inf\{t > 0: \mu_N(t) \notin A\}$, which denotes the time exit from the set $A$. Finally, for a subset $W \subset L$, we define the stopping time $\hat{\tau}_W \coloneqq \inf\{t >0 : \mu_N(t) \in \cup_{i \in W}\gamma_i\}$, and $\bar{\tau}_W \coloneqq \inf \{t > 0: \mu_N(t) \in \cup_{i \in L \setminus W} \gamma_i \}$, which denote the time of entry into the $\rho_1$-neighbourhood of $W$ and the time of entry into the $\rho_1$-neighbourhood of $L \setminus W$, respectively.

We now state some results on the behaviour of the exit time from certain sets, which will be used in the paper subsequently. These results are known in the case of both Markov jump processes as well as diffusion processes; see~{\cite[Appendix]{borkar-sundaresan-12}}, and~{\cite[Chapter~6,~Section~2]{freidlin-wentzell}}. The main ingredients that are used in proving these results are (i) the strong Markov property of the $\mu_N$ process, (ii) Theorem~\ref{thm:finite-duration-ldp} and Corollary~\ref{cor:uniform_ldp} on the LDP for finite time durations, and (iii) the joint continuity of the terminal time rate function $S_T(\cdot|\cdot)$ (see~{\cite[Lemma~3.3]{borkar-sundaresan-12}}). Recall that $P_\nu$ denotes the law of $(\mu_N(t), t \geq 0)$ with initial condition $\mu_N(0) = \nu$ and $E_\nu$ denotes the corresponding expectation.

\begin{lemma}[{\cite[Lemma A.3]{borkar-sundaresan-12}}]
Let $K \subset M_1(\mathcal{Z})$  be a compact set such that all points in $K$ are equivalent to each other. Then, given $\varepsilon> 0$, there exist $\delta > 0$ and $N_0 \geq 1$ such that for all $N \geq N_0$ and $\nu \in  [K]_{\delta}\cap M^N_1(\mathcal{Z})$,
\begin{align*}
E_\nu \tau_{[K]_{\delta}} \leq \exp\{N\varepsilon\}.
\end{align*}
\label{lemma:fw17}
\end{lemma}

\begin{lemma}[{\cite[Lemma A.3]{borkar-sundaresan-12}}]
Let $K \subset M_1(\mathcal{Z})$ be a compact set and $G$ be a neighbourhood of $K$. Then, given $\varepsilon > 0$, there exist $\delta>0$ and $N_0 \geq 1$ such that for all  $\nu \in \overline{[K]_{\delta}} \cap M_1^N(\mathcal{Z})$ and $N \geq N_0$
\begin{align*}
E_\nu \left( \int_0^{\tau_G}  \indf{\{\mu_N(t) \in \overline{[K]_\delta}\}} dt \right) \geq \exp\{-N\varepsilon\}.
\end{align*}
\label{lemma:fw18}
\end{lemma}

\begin{lemma}[{\cite[Lemma A.5]{borkar-sundaresan-12}}]
Let $K \subset M_1(\mathcal{Z})$ be a compact set that does not contain any $\omega$-limit set of~(\ref{eqn:MVE}) entirely. Then, there exist positive constants $c, T_0$ and $N_0 \geq 1$ such that for all $T \geq T_0$, $N \geq N_0$ and any $\nu \in K \cap M_1^N(\mathcal{Z})$, we have
\begin{align*}
P_{\nu}( \tau_K \geq T ) \leq \exp\{-Nc(T-T_0)\}.
\end{align*}
\label{lemma:fw19}
\end{lemma}
\begin{corollary}
Under the conditions of Lemma~\ref{lemma:fw19}, there exist $ C>0$ and $N_0 \geq 1$ such that for all $\nu \in K \cap M_1^N(\mathcal{Z})$ and $N \geq N_0$,
\begin{align*}
E_\nu \tau_K \leq C.
\end{align*}
\label{cor:fw19}
\end{corollary}
Recall the definition of the discrete time Markov chain $Z^N$ on $\gamma \cap M_1^N(\mathcal{Z})$. The next lemma gives upper and lower bounds on the one-step transition probabilities of the chain $Z^N$. These estimates play an important role in the study of large-time behaviour of the process $\mu_N$, as we shall see in the sequel.
\begin{lemma}[{\cite[Lemma A.6]{borkar-sundaresan-12}}]
\label{lemma:bsa6}
Given $\varepsilon >0$, there exist $\rho_0 >0$ and $N_0 \geq 1$ such that, for any $\rho_2  < \rho_0$, there exists $\rho_1 < \rho_2$ such that for any $\nu \in [K_i]_{\rho_2} \cap M_1^N(\mathcal{Z})$ and $N \geq N_0$, the one-step transition probability of the chain $Z^N$ satisfies
\begin{align}
\exp\{-N(\tilde{V}(K_i, K_j)+\varepsilon)\} \leq P(\nu, \gamma_j) \leq \exp\{-N(\tilde{V}(K_i, K_j)-\varepsilon)\}.
\label{eqn:tpm_zn}
\end{align}
\end{lemma}
\begin{remark}
In the above statement, $P(\nu, \gamma_j)$ is defined as $P(\nu, \gamma_j) \coloneqq P_\nu(Z^N_1 \in \gamma_j) = P_\nu(\mu_N(\tau_1) \in \gamma_j)$.
\end{remark}
The key ingredient in the proof of the above lemma is Corollary~\ref{cor:uniform_ldp} on the uniform large deviation principle on bounded sets. For the lower bound, one constructs a specific trajectory from $\nu$ to $K_j$ and examines its cost. For the upper bound, one uses the strong Markov property at the hitting time of $[L]_{\rho_1}$ and the uniform large deviation principle. For details, the reader is referred to proof of~{\cite[Lemma~A.6]{borkar-sundaresan-12}} for the case of Markov jump processes, and proof of~\cite[Lemma~2.1,~page~152]{freidlin-wentzell} for the case of small noise diffusions.

\subsection{Behaviour near attractors indexed by subsets of $L$}
\label{subsection:L-subsets}
We now recall some results on the behaviour of the process $\mu_N$ near a small neighbourhood of attractors indexed by a given subset of $L$. Let $W \subset L$. A $W$-graph is a directed graph on $L$ such that (i) each element of $L \setminus W$ has exactly one outgoing arrow and (ii) there are no closed cycles in the graph. We denote the set of $W$-graphs by $G(W)$. For a $W$-graph $g$, define $\tilde{V}(g) = \sum_{(m\to n) \in g} \tilde{V}(K_m, K_n)$. Note that, using the estimate~(\ref{eqn:tpm_zn}), $\tilde{V}$ can be used to estimate the probability that the process $\mu_N$ traverses through a sequence of neighbourhoods in the order specified by the graph $g$.

For $i \in L\setminus W$ and $j \in W$, let $G_{i,j}(W)$ denote the set of $W$-graphs in which there is a sequence of arrows leading from $i$ to $j$.
Define
\begin{align*}
I_{i,j}(W) \coloneqq \min\{\tilde{V}(g):g \in G_{i,j}(W)\} - \min\{\tilde{V}(g): g \in G(W)\}.
\end{align*}
We recall the following result on the probability that the first entry of $\mu_N$ into a neighbourhood of a set $W \subset L$ takes place via a given compact set $K_j$, starting from a neighbourhood of $K_i$.
\begin{lemma}
Let $W \subset L$, and let $i \in L \setminus W$ and $j \in W$. Given $\varepsilon > 0$, there exist $\rho > 0$ and $N_0 \geq 1 $ such that for any $\rho_1 \leq \rho$, $\nu \in \gamma_i \cap M^N_1(\mathcal{Z})$ and $N \geq N_0$, we have
\begin{align*}
 \exp\{-N(I_{i,j}(W)+\varepsilon)\} \leq P_\nu(\mu_N(\hat{\tau}_W) \in \gamma_j) \leq \exp\{-N(I_{i,j}(W)-\varepsilon)\}.
\end{align*}
\label{lemma:hs15}
\end{lemma}
\begin{proof}The proof of~{\cite[Lemma 3.3, page 159]{freidlin-wentzell}} holds verbatim, by making use of the estimates in~Lemma~\ref{lemma:bsa6}.
\end{proof}
\begin{remark}While the above lemma provides an estimate of the probability $P_\nu(\mu_N(\hat{\tau}_W) \in \gamma_j)$, it does not provide any information about the sequence of states in $L$ visited by the process $\mu_N$ while traversing from $i$ to $j$. The latter can be understood via studying the minimisations in the definition of $I_{i,j}$, see~\cite{gan-cameron-17}.
\end{remark}

Our next step is to understand the mean entry time $E_\nu \hat{\tau}_W$. For this, we need the following estimate on the stopping time $\tau_1$; see~{\cite[Lemma~1.3, Part I]{hwang-sheu-90}} for a similar estimate for small noise diffusion processes.
\begin{lemma}Given $\varepsilon >0$, there exist $\rho_1 >0$ and $N_0 \geq 1$ such that, for any $\nu \in \gamma \cap M_1^N(\mathcal{Z})$ and $N \geq N_0$, we have
\begin{align*}
E_\nu \tau_1 \leq \exp\{N\varepsilon\}.
\end{align*}
\label{lemma:hs13}
\end{lemma}
\begin{proof}
With a sufficiently small $\rho_1 > 0$ to be chosen later, let $\rho_0 = 2 \rho_1$ so that $[K_i]_{\rho_0}$ does not intersect with $[K_j]_{\rho_0}$ for all $j \neq i$. Note that, for any $\nu \in \gamma$,
\begin{align*}
E_{\nu} \tau_1 = E_\nu \sigma_0  +E_\nu (\tau_1 - \sigma_0).
\end{align*}
Consider the first term. By Lemma~\ref{lemma:fw17},  there exist $\rho > 0$ and $N_0 \geq 1$ such that for all $\rho_1 \leq \rho$, $\nu \in \gamma \cap M_1^N(\mathcal{Z})$ and $N \geq N_0$, we have
\begin{align*}
E_\nu \sigma_0 \leq \exp\{N\varepsilon/2\}.
\end{align*}
Let $F = M_1(\mathcal{Z}) \setminus \gamma$. By the strong Markov property, the second term is
\begin{align*}
E_\nu (\tau_1 - \sigma_0) = E_{\mu_N(\sigma_0)} (\tau_{F}).
\end{align*}
Therefore, it suffices to estimate $E_{\nu^\prime} \tau_F$ for $\nu^\prime \in F$. Since the compact set $F$ does not contain any $\omega$-limit set, by Corollary~\ref{cor:fw19}, there exist a constant $C >0$ and $N_1 \geq N_0$ such that for any $\nu^\prime \in F \cap M_1^N(\mathcal{Z})$
\begin{align*}
E_{\nu^\prime} \tau_{F} \leq C.
\end{align*}
This completes the proof of the lemma.
\end{proof}
Define
\begin{align*}
I_i(W) \coloneqq &\min\{\tilde{V}(g): g \in G(W)\} -  \min\{\tilde{V}(g): g \in G( W \cup \{i\}) \text{ or } \\&  g \in G_{i,j}(W \cup \{j\}), i \neq j, j \in L\setminus W\} \
\end{align*}
The next lemma is about the mean entry time into a neighbourhood of a given set $W \subset L$ starting from a neighbourhood of  $K_i$; see~{\cite[Lemma~1.6, Part I]{hwang-sheu-90}} for a similar estimate on small noise diffusion processes.

\begin{lemma}
\label{lemma:hs16}
Let $W \subset L$, and let $i \in L \setminus W$. Given $\varepsilon > 0$, there exist $\rho > 0$ and $N_0 \geq 1 $ such that for any $\rho_1 \leq \rho$, $\nu \in \gamma_i \cap M^N_1(\mathcal{Z})$ and $N \geq N_0$, we have
\begin{align*}
\exp\{N(I_i(W)-\varepsilon)\} \leq E_\nu \hat{\tau}_W \leq \exp\{N(I_i(W)+\varepsilon)\}.
\end{align*}
\end{lemma}
\begin{proof}
We first prove the upper bound. Note that, by the strong Markov property, we have
\begin{align*}
E_\nu \hat{\tau}_W = E_\nu \tau_v \leq \sum_{m=1}^\infty E_\nu \left(\indf{v = m} \times m \sup_{\nu^\prime \in \gamma} E_{\nu^\prime} \tau_1 \right),
\end{align*}
where $v$ is the hitting time of the chain $Z_n^N$ on the set $W$. Using~Lemma~\ref{lemma:hs13} and the upper bound on $E_\nu v$ derived in~{\cite[Lemma~3.4,~page~162]{freidlin-wentzell}}, for sufficiently small $\rho_1$ and sufficiently large $N$, we have that
\begin{align*}
E_\nu \hat{\tau}_W \leq \exp\{N(I_i(W) +\varepsilon)\}
\end{align*}
holds for all $\nu \in \gamma_i \cap M_1^N(\mathcal{Z})$. For the lower bound,  Lemma~\ref{lemma:fw18} implies that, for all sufficiently small $\rho_1$ and sufficiently large $N$, we have that
\begin{align*}
E_\nu \tau_1 \geq \exp\{-N\varepsilon\}
\end{align*}
holds for all $\nu \in \gamma$. Also,
\begin{align*}
E_\nu \hat{\tau}_W = E_\nu \tau_v \geq  \sum_{m=1}^\infty E_\nu \left(1_{v = m} \times m \inf_{\nu^\prime \in \gamma} E_{\nu^\prime} \tau_1 \right),
\end{align*}
hence, using the lower bound on $E_\nu v$ derived in~{\cite[Lemma~3.4,~page~162]{freidlin-wentzell}}, we get
\begin{align*}
E_\nu \hat{\tau}_W \geq \exp\{N(I_i(W) - \varepsilon)\}
\end{align*}
for all $\nu \in \gamma_i \cap M_1^N(\mathcal{Z})$ and sufficiency large $N$.
\end{proof}
\subsection{Cycles}
\label{subsection:cycles}
We now define the notion of cycles, which helps us to describe the most probable way in which the process $\mu_N$, for large $N$, traverses neighbourhoods of various compact sets $K_i$, and the time required to go from one to another. Define $\tilde{V}(K_i) \coloneqq \min_{j \neq i} \tilde{V}(K_i, K_j)$. We say that $i \to j$ if $\tilde{V}(K_i) = \tilde{V}(K_i,K_j)$. Note that, using the estimates~(\ref{eqn:tpm_zn}) on the transition probability of the discrete time Markov chain $Z^N$, we see that the indices that attain the minimum above are the most likely sets that will be visited by the process $\mu_N$, for large enough $N$, starting from a neighbourhood of $K_i$. For $i, j \in L$, we say that $i \Rightarrow j$ if there exists a sequence of arrows leading from $i$ to $j$, i.e., there exists $i_1, i_2 , \ldots, i_n$ in $L$ such that $i \to i_1 \to i_2 \to \cdots \to i_n \to j$. Again, the above sequence of arrows from $i$ to $j$ is one among the \emph{locally} most likely sequences in which the process traverses from a neighbourhood of $K_i$ to that of $K_j$ for large $N$.
\begin{definition}
A cycle $\pi$ is a directed graph on a subset of elements of $L$ satisfying
\begin{enumerate}
\item $i \in \pi$ and $i \Rightarrow j$ implies $j \in \pi$.
\item For any $i \neq j$ in $\pi$, we have $i \Rightarrow j $ and $j \Rightarrow i$.
\end{enumerate}
\end{definition}
It can be shown that there exists a cycle (see the proof of~\cite[Lemma~1.9,~Part~I]{hwang-sheu-90}). We now define cycle of cycles. Let $L_0 = L$. Define
\begin{align*}
L_1 \coloneqq \{\pi: \pi \text{ is a cycle in } L\} \cup \{i \in L: i \text{ is not in any cycle}\}.
\end{align*}
For $\pi_1, \pi_2 \in L_1$, $\pi_1 \neq \pi_2$, define
\begin{align*}
\hat{V}(\pi_1) \coloneqq \max \{\tilde{V}(K): K \in \pi_1\},
\end{align*}
\begin{align*}
\tilde{V}(\pi_1, \pi_2) \coloneqq \hat{V}(\pi_1) + \min \{\tilde{V}(K_1, K_2)-\tilde{V}(K_1): K_1 \in \pi_1, K_2 \in \pi_2\},
\end{align*}
and
\begin{align*}
\tilde{V}(\pi_1) \coloneqq \min\{\tilde{V}(\pi_1, \pi_2): \pi_2 \in L_1, \pi_2 \neq \pi_1\}.
\end{align*}
We say that $\pi_1 \to \pi_2$ if $ \tilde{V}(\pi_1) = \tilde{V}(\pi_1, \pi_2)$, and we say that $\pi_1 \Rightarrow \pi_2$ if there is a sequence of arrows leading from $\pi_1$ to $\pi_2$.  This gives a cycle of cycles, which we call 2-cycles.

Let us now define the hierarchy of cycles. Having defined  $(m-1)$-cycles and the sets $L_0, L_1, \ldots, L_{m-2}$, we define $m$-cycles as follows. Note that
\begin{align*}
L_{m-1}&  = \{\pi^{m-1}: \pi^{m-1} \text{ is an }  (m-1)\text{-cycle}\}   \\
& \qquad \cup \{\pi^{m-2} \in L_{m-2} : \pi^{m-2} \text{ is not in any } (m-1)\text{-cycle}\}.
\end{align*}
For $\pi^{m-1} \in L_{m-1}$, define
\begin{align*}
\hat{V}(\pi^{m-1}) \coloneqq  \max \{\tilde{V}(\pi^{m-2}): \pi^{m-2} \in \pi^{m-1}\},
\end{align*}
\begin{align*}
\tilde{V}(\pi_1^{m-1}, \pi_2^{m-1}) & \coloneqq \hat{V}(\pi_1^{m-1}) +   \min \{\tilde{V}(\pi_1^{m-2}, \pi_2^{m-2})-\tilde{V}(\pi_1^{m-2}) \\
& \qquad : \pi_1^{m-2} \in \pi_1^{m-1}, \pi_2^{m-2} \in \pi_2^{m-1}\},
\end{align*}
and
\begin{align*}
\tilde{V}(\pi_1^{m-1}) \coloneqq \min\{\tilde{V}(\pi_1^{m-1}, \pi_2^{m-1}): \pi_2^{m-1} \in L_{m-1}, \pi_2^{m-1} \neq \pi_1^{m-1}\}.
\end{align*}
We say that $\pi_1^{m-1} \to \pi_2^{m-1}$ if $\tilde{V}(\pi_1^{m-1}) = \tilde{V}(\pi_1^{m-1}, \pi_2^{m-1})  $. We have
\begin{definition} An $m$-cycle $\pi^m$ is a directed graph on a subset of elements of $L_{m-1}$ satisfying
\begin{enumerate}
\item For $\pi_1^{m-1}, \pi_2^{m-1} \in L_{m-1}$, $\pi_1^{m-1} \in \pi^m$ and $\pi_1^{m-1} \Rightarrow \pi_2^{m-1}$ implies $\pi_2^{m-1} \in \pi^m$.
\item For any $\pi_1^{m-1}, \pi_2^{m-1} \in \pi^m$, we have $\pi_1^{m-1} \Rightarrow \pi_2^{m-1}$ and $\pi_2^{m-1}\Rightarrow \pi_1^{m-1}$.
\end{enumerate}
\end{definition}
If we continue this way, for some $m \geq 1$, the set $L_m$ will eventually be a singleton, at which point we stop.

We now state some results on the mean exit time from a cycle and the most probable cycle the process $\mu_N$ visits upon exit from a given cycle. For  convenience, the set of elements of $L$ constituting a $k$-cycle $\pi^k$ (through the hierarchy of cycles) is also denoted by $\pi^k$. Also, for $W \subset L$, we define $\gamma_W = \cup_{i \in W} \gamma_i$.
\begin{corollary}
Let $\pi^k$ be a $k$-cycle and $K_i \in \pi^k$. Let $W = L \setminus \pi^k$. Given $\varepsilon > 0$, there exist $\rho >0$ and $N_0 \geq 1$ such that for all $\rho_1 \leq \rho$, $\nu \in \gamma_i \cap M_1^N(\mathcal{Z})$ and $N \geq N_0$, we have
\begin{align*}
\exp \{N(\tilde{V}(\pi^k) - \varepsilon)\} \leq E_\nu \hat{\tau}_W \leq \exp \{N(\tilde{V}(\pi^k) + \varepsilon)\}.
\end{align*}
\label{cor:hs110}
\end{corollary}

\begin{corollary}
Let $\pi_1^k, \pi_2^k$ be $k$-cycles, $\pi_1^k \neq \pi_2^k$, and $K_i \in \pi_1^k$. Let $W = L \setminus \pi_1^k$. Given $\varepsilon > 0$, there exist $\rho >0$ and $N_0 \geq 1$ such that for all $\rho_1 \leq \rho$, $\nu \in \gamma_i \cap M_1^N(\mathcal{Z})$ and $N \geq N_0$, we have
\begin{align*}
\exp \{-N(\tilde{V}(\pi_1^k, \pi_2^k) - \tilde{V}(\pi_1^k) + \varepsilon)\} & \leq  P_\nu (\mu_N(\hat{\tau}_W) \in \gamma_{\pi_2^k})  \\
&\leq \exp \{-N(\tilde{V}(\pi_1^k, \pi_2^k) - \tilde{V}(\pi_1^k) - \varepsilon)\}.
\end{align*}
\label{cor:hs111}
\end{corollary}
\begin{remark}
Note that Corollary~\ref{cor:hs110} follows from Lemma~\ref{lemma:hs16} and the fact that $I_i(W) = \tilde{V}(\pi^k)$ (which is shown in~\cite[Corollary~A.4,~Appendix]{hwang-sheu-90}). Corollary~\ref{cor:hs111} is a consequence of Lemma~\ref{lemma:hs15} along with the fact that $\min\{I_{i,j}(W) : i \in \hat{\pi}^k\} = \tilde{V}(\pi^k, \hat{\pi}^k) - \tilde{V}(\pi^k)$ (see~ \cite[Corollary~A.6,~Appendix]{hwang-sheu-90}). Similar estimates as in  Corollaries~\ref{cor:hs110} and~\ref{cor:hs111} in the case of small noise diffusion processes have been shown in {\cite[Corollary~1.10, Part I]{hwang-sheu-90}} and {\cite[Corollary~1.11, Part I]{hwang-sheu-90}}, respectively.
\end{remark}

We also need the following lemmas that provide estimates on the probabilities of exit within certain times from given cycles.
\begin{lemma}
Let $\pi_1^k, \pi_2^k$ be $k$-cycles and let $\pi_1^k \to \pi_2^k$. Then, given $\varepsilon >0$, there exist $\delta >0$, $\rho >0$ and $N_0 \geq 1$ such that for all $\rho_1 \leq \rho$, $\nu \in \gamma_{\pi_1^k} \cap M_1^N(\mathcal{Z})$ and $N \geq N_0$, we have
\begin{align*}
P_\nu \left(\bar{\tau}_{\pi_1^k} \leq \exp\{N(\tilde{V}(\pi_1^k)-\delta)\}, \mu_N(\bar{\tau}_{\pi_1^k}) \in \gamma_{\pi_2^k} \right) \geq \exp\{-N\varepsilon\}.
\end{align*}
\label{lemma:hitting_place}
\end{lemma}
\begin{lemma}
Let $\pi^k$ be a $k$-cycle. Then, given $\varepsilon > 0$, there exists $\rho > 0$ such that for all $\rho_1 \leq \rho$, we have
\begin{align*}
\lim_{N \to \infty} \sup_{\nu \in \gamma_{\pi^k} \cap M_1^N(\mathcal{Z})} P_\nu\left(  \exp\{N(\tilde{V}(\pi^k) - \varepsilon)\} \leq \bar{\tau}_{\pi^k} \leq \exp\{N(\tilde{V}(\pi^k) + \varepsilon)\} \right) =1.
\end{align*}
Furthermore, given $\varepsilon > 0$, there exist $\delta >0$, $\rho > 0$ and $N_0 \geq 1$ such that for all $\rho_1 \leq \rho$, $N \geq N_0$ and $\nu \in \gamma_{\pi^k} \cap M^N_1(\mathcal{Z})$, we have
\begin{align*}
P_\nu\left(  \bar{\tau}_{\pi^k}  < \exp\{N(\tilde{V}(\pi^k) - \delta)\} \right) \leq \exp\{-N\varepsilon\}, \text{ and} \\
P_\nu\left(  \bar{\tau}_{\pi^k}  > \exp\{N(\tilde{V}(\pi^k) + \delta)\} \right) \leq \exp\{-N\varepsilon\}.
\end{align*}
\label{lemma:hitting_time}
\end{lemma}
\begin{remark}
Lemma~\ref{lemma:hitting_place} can be proved using Lemma~\ref{lemma:fw19} and {\cite[Chapter~6,~Theorem~6.2]{freidlin-wentzell}}, and Lemma~\ref{lemma:hitting_time} can be proved using the same arguments used in the proof of~{\cite[Chapter~6,~Theorem~6.2]{freidlin-wentzell}}. Similar estimates as in Lemmas~\ref{lemma:hitting_place} and~\ref{lemma:hitting_time} in the case of small noise diffusion processes have been shown in {\cite[Lemma 2.1, Part I]{hwang-sheu-90}} and {\cite[Lemma 2.2, Part I]{hwang-sheu-90}}, respectively.
\end{remark}
\begin{lemma}
Let $\pi^k$ be a $k$-cycle and assume that $\tilde{V}(\pi^k) > 0$. Given $\varepsilon > 0$, there exist $\delta>0, \rho >0$ and $N_0 \geq 1$ such that for all $\rho_1 \leq \rho, \nu \in M_1^N(\mathcal{Z})$ and $N \geq N_0$, we have
\begin{align*}
P_{0,\nu}(\bar{\tau}_{\pi^k} \leq \exp\{N(\hat{V}(\pi^k)+\delta\}) \leq \exp\{-N(\tilde{V}(\pi^k) - \hat{V}(\pi^k) - \varepsilon)\}.
\end{align*}
\label{lemma:exit_time_vhat}
\end{lemma}
\begin{proof}
We proceed via the steps in the proof of~{\cite[Lemma~2.1, Part III]{hwang-sheu-90}}. Let $\pi^{k-1} \in \pi^k$ be a $(k-1)$-cycle such that $\tilde{V}(\pi^{k-1}) = \hat{V}(\pi^k)$. With $\rho_1>0$ to be chosen later, for each $n \geq 1$, define the minimum of $\bar{\tau}_{\pi^k}$ and successive entry and exit times from a $\rho_1$-neighbourhood of $\pi^{k-1}$ as follows:
\begin{align*}
\hat{\theta}_0 & \coloneqq \inf\{ t >0 : \mu_N(t) \in [\pi^{k-1}]_{\rho_1}\} \wedge \bar{\tau}_{\pi^k}, \\
\bar{\theta}_n & \coloneqq \inf\{t > \hat{\theta}_{n-1}: \mu_N(t) \in [L \setminus \pi^{k-1}]_{\rho_1} \} \wedge \bar{\tau}_{\pi^k} , \\
\hat{\theta}_{n+1}&  \coloneqq \inf\{t > \bar{\theta}_n : \mu_N(t) \in  [\pi^{k-1}]_{\rho_1} \} \wedge \bar{\tau}_{\pi^k}.
\end{align*}
With $\delta > 0$ to be chosen later, using the strong Markov property, for any $\nu \in [\pi^k]_{\rho_1} \cap M_1^N(\mathcal{Z})$, we have
\begin{align}
P_{\nu}&(\bar{\tau}_{\pi^k}   \leq \exp\{N(\hat{V}(\pi^k)+\delta)\})  = P_\nu (\hat{\theta}_0 = \bar{\tau}_{\pi^k}, \bar{\tau}_{\pi^k} \leq \exp\{N(\hat{V}(\pi^k)+\delta)\}) \nonumber \\
 + & P_\nu \left(\hat{\theta}_0 < \bar{\tau}_{\pi^k}, \bigcup_{n \geq 1} \left\{\bar{\tau}_{\pi^k} = \bar{\theta}_n, \bar{\tau}_{\pi^k} \leq \exp\{N(\hat{V}(\pi^k)+\delta)\}, \bar{\tau}_{\pi^k} \geq \hat{\theta}_{n-1} \right\} \right) \nonumber \\
+ & P_\nu \left(\hat{\theta}_0 < \bar{\tau}_{\pi^k}, \bigcup_{n\geq 1} \left\{\bar{\tau}_{\pi^k} = \hat{\theta}_n, \bar{\tau}_{\pi^k} \leq \exp\{N(\hat{V}(\pi^k)+\delta)\}, \bar{\tau}_{\pi^k} \geq \bar{\theta}_n \right\}  \right).
\label{eqn:temp10}
\end{align}
We now upper bound each of the terms in~\ref{eqn:temp10}.Consider the first term. It can be shown using Corollary~\ref{cor:hs111} and~\cite[Corollary~A.6,~Appendix]{hwang-sheu-90} that, there exist $\rho_1>0$ and $\delta > 0$ such that for any $\nu \in [\pi^k]_{\rho_1}$ and sufficiently large $N$, we have
\begin{align*}
P_\nu(\hat{\theta}_0 = \bar{\tau}_{\pi^k}) \leq \exp\{-N(\tilde{V}(\pi^k) - \hat{V}(\pi^k) - \varepsilon)\}.
\end{align*}

Consider the second term in~\ref{eqn:temp10}. For any $\nu_1 \in [\pi^{k-1}]_{\rho_1} \cap M_1^N(\mathcal{Z})$, the probability of the unionised event can be upper bounded by
\begin{align*}
P_{\nu_1} & \left(\bigcup_{n \geq 1}\left\{\bar{\tau}_{\pi^k} = \bar{\theta}_n, \bar{\tau}_{\pi^k} \leq \exp\{N(\hat{V}(\pi^k)+\delta)\}, \bar{\tau}_{\pi^k} \geq \hat{\theta}_{n-1} \right\} \right)\\
& \leq P_{\nu_1} \left(\bigcup_{n=1}^M \left\{\bar{\tau}_{\pi^k} = \bar{\theta}_n, \bar{\tau}_{\pi^k} \leq \exp\{N(\hat{V}(\pi^k)+\delta)\}, \bar{\tau}_{\pi^k} \geq \hat{\theta}_{n-1} \right\} \right) \\
& \, \, \, \, + P_{\nu_1} \left( \bigcup_{n \geq M+1}  \left\{\bar{\tau}_{\pi^k} = \bar{\theta}_n, \bar{\tau}_{\pi^k} \leq \exp\{N(\hat{V}(\pi^k)+\delta)\}, \bar{\tau}_{\pi^k} \geq \hat{\theta}_{n-1} \right\} \right) \\
& \leq P_{\nu_1} (\bar{\tau}_{\pi^k} = \bar{\theta}_n \text{ and } \bar{\tau}_{\pi^k} \geq \hat{\theta}_{n-1} \text{ for some } n \leq M) \\
&  \, \, \, \, + P_{\nu_1} (\hat{\theta}_M \leq \exp\{N(\hat{V}(\pi^k)+\delta)\} \text{ and } \hat{\theta}_M \leq \bar{\tau}_{\pi^k})\\
& \leq P_{\nu_1}(\hat{\theta}_M = \bar{\tau}_{\pi^k}) + P_{\nu_1} (\hat{\theta}_M \leq \exp\{N(\hat{V}(\pi^k)+\delta)\} \text{ and } \hat{\theta}_M \leq \bar{\tau}_{\pi^k}).
\end{align*}
Again, the first term above can be bounded by
\begin{align*}
P_{\nu_1}(\hat{\theta}_M \leq \bar{\tau}_{\pi^k}) \leq \exp\{-N(\tilde{V}(\pi^k) - \hat{V}(\pi^k) - \varepsilon)\},
\end{align*}
for all $\nu_1 \in [\pi^{k-1}]_{\rho_1} \cap M_1^N(\mathcal{Z})$ and sufficiently large $N$.  The second term can be bounded by $\exp\{-NM\}$ for large enough $M$, by the same argument used in the proof of~\cite[Lemma~1.7,~Part~I]{hwang-sheu-90}. Choosing $M$ sufficiently large, the above implies that the second term in~(\ref{eqn:temp10}) is bounded by $\exp\{-N(\tilde{V}(\pi^k) - \hat{V}(\pi^k) - \varepsilon)\}$. A similar argument gives the same bound for the third term in~(\ref{eqn:temp10}).
\end{proof}
\subsection{LDP for the invariant measure}
Using the estimates~(\ref{eqn:tpm_zn}) of the transition probabilities of the discrete time Markov chain $Z^N$, we can study large deviations for the process $\mu_N$ in the stationary regime. Recall that $\wp_N$ denotes the unique invariant probability measure of the process $\mu_N$. We state the following result:
\begin{theorem}[{\cite[Theorem~2.2]{borkar-sundaresan-12}}]
Assume~\ref{assm:a1},~\ref{assm:a2} and~\ref{assm:b1}. Then, the sequence of invariant measures $\{\wp_N\}_{N \geq 1}$ satisfies the large deviation principle on $M_1(\mathcal{Z})$ with good rate function $s$ given by
\begin{align}
s(\xi) = \min_{1 \leq i \leq l} \{W(i) + V(K_i, \xi)\} - \min_{1 \leq j \leq l} W(j),
\label{eqn:s_expression},
\end{align}
where
\begin{align*}
W(i) = \min_{g \in G(i)} \sum_{(m,n) \in g} \tilde{V}(m,n).
\end{align*}
\label{thm:invariant-measure-gase}
\end{theorem}

The form of the rate function $s$ in Theorem~\ref{thm:invariant-measure-gase} is also related to the form of the invariant measure in the context of Markov chains on finite state spaces whose transition kernels are of the form~(\ref{eqn:tpm_zn}); see, for example,~\cite[Section~1.1]{delmoral-miclo-99}. Also, see~\cite{bodineau-giacomin-04} for an analogous result in a boundary driven symmetric simple exclusion process, which involves the study of the LDP for the invariant measure in an infinite dimensional setting. However, our focus is on sharp estimates on the rate of convergence to the invariant measure which is the subject of the next section.

\subsection{Convergence to the invariant measure}
\label{subsection:convergence-invariant-measure}
In this section, we prove our first main result on the time required for the convergence of $\mu_N$ to its invariant measure. 

Let $i_0 \in L$ be such that $ \min\{\tilde{V}(g): g \in G(i_0)\} = \min\{\tilde{V}(g): g \in G(i), i \in L \} $. We anticipate that $K_{i_0}$ is one of the most stable $\omega$-limit sets (among possibly others) for the dynamics~\eqref{eqn:MVE}. This is because Theorem~\ref{thm:invariant-measure-gase} tells us that the rate function that governs the LDP for $\{\wp_N\}_{N \geq 1}$ vanishes on $K_{i_0}$. Hence, for a large but fixed $N$, over large time intervals, one expects that there is positive probability (in the exponential scale) for the process $\mu_N$ to be in a small neighbourhood of $K_{i_0}$.

Define
\begin{align}
\Lambda \coloneqq \min\{\tilde{V}(g): g \in G(i), i \in L \} - \min\{\tilde{V}(g): g \in G(i, j), i, j \in L, i \neq j\}.
\label{eqn:lambda-def}
\end{align}
Let $P_T(\nu, \cdot ) = P_{\nu}(\mu_N(T)  \in \cdot )$ denote the transition probability kernel associated with the process $\mu_N$. Note that we suppress the dependence on $N$ for ease of readability. We first show  a lower bound for the transition probability $P_T(\nu_1, K_{i_0} ) $ of reaching a small neighbourhood of $K_{i_0}$ when $T$ is of the order $\exp\{N(\Lambda-\delta_0)\}$ for some $\delta_0>0$.

\begin{theorem}
Given $\varepsilon >0$, there exist $\delta_0 > 0$, $\rho > 0$ and $N_0 \geq 1$ such that for all $\rho_1 \leq \rho$, $N \geq N_0$, $\nu \in M_1^N(\mathcal{Z})$, we have
\begin{align}
P_{T_0}(\nu, \gamma_{i_0}) \geq \exp\{-N\varepsilon\},
\label{eqn:lb_tpm_i0}
\end{align}
where $T_0 =\exp\{N(\Lambda - \delta_0)\}$. Furthermore, there exist $\nu_0 \in M_1(\mathcal{Z})$ and $\beta >0$ such that for all $N \geq N_0$ and $\nu \in [\nu_0]_{\rho_1} \cap M_1^N(\mathcal{Z})$
\begin{align}
P_{T_0}(\nu, \gamma_{i_0}) \leq \exp\{-N\beta\}.
\label{eqn:ub_tpm}
\end{align}
\label{thm:mixing}
\end{theorem}
\begin{proof}
We follow the steps in Hwang and Sheu~{\cite[Part I, Theorem~2.3]{hwang-sheu-90}}. With $\rho > 0$ to be chosen later, we first show that~(\ref{eqn:lb_tpm_i0}) holds for all $\nu \in \gamma \cap M_1^N(\mathcal{Z})$. Towards this, let $m$ be the smallest integer such that $L_{m+1}$ is a singleton. For $0 \leq k \leq m$, let $\pi_0^k \in L_k$  be the $k$-cycle containing $i_0$. Let $V_k = \max \{\tilde{V}(\pi^k) : \pi^k \subset \pi_0^{k+1}, \pi^k \neq \pi_0^k\}$. Using~\cite[Lemma~A.10,~Appendix]{hwang-sheu-90}, we have $\Lambda = \max\{V_k: 0 \leq k \leq m\}$.

Fix $j \in L$ and consider $\nu \in [K_j]_{\rho}$. Let $\pi_1^m \in L_m$ be such that $K_j \in \pi_1^m$. If $\pi_1^m \neq \pi_0^m$, then we have $\pi_1^m \Rightarrow \pi_0^m$, that is,  there exists $\pi_2^m, \pi_3^m, \ldots, \pi_n^m =\pi_0^m, n \leq l$ such that $\pi_1^m \to \pi_2^m \to \pi_3^m \to \cdots \to \pi_n^m = \pi_0^m$. Therefore, with $\delta$ to be chosen later, by the strong Markov property (we use the standard notation $E_\nu(A; B)$ for $E_\nu(1_A 1_B)$ where $A$ and $B$ are measurable sets),
\begin{align*}
P_\nu (\hat{\tau}_{\pi^m_0} & \leq n \exp\{N(V_m - \delta)\})  \\
& \geq E_\nu ( \bar{\tau}_{\pi^m_1} \leq  \exp\{N(V_m -\delta )\}, \mu_N(\bar{\tau}_{\pi^m_1}) \in \pi_2^m;  \\
& \, \, \, \, \, \, \, \, E_{\mu_N(\bar{\tau}_{\pi^m_1})}(\bar{\tau}_{\pi_2^m} \leq  \exp\{N(V_m -\delta )\},\mu_N(\bar{\tau}_{\pi^m_2}) \in \pi_3^m; \\
& \, \, \, \, \, \, \, \, \cdots  E_{\mu_N(\bar{\tau}_{\pi^m_{n-2}})}( \bar{\tau}_{\pi^m_{n-1}} \leq  \exp\{N(V_m -\delta )\},\mu_N(\bar{\tau}_{\pi^m_{n-1}}) \in \pi_0^m)\\
& \, \, \, \, \, \, \, \,  \cdots )).
\end{align*}
Since $V(\pi_i^m) \leq V_m$ for all $1 \leq i \leq n$, the above becomes

\begin{align*}
P_\nu (\hat{\tau}_{\pi^m_0} & \leq n \exp\{N(V_m - \delta)\}) \\
&  \geq E_\nu ( \bar{\tau}_{\pi^m_1} \leq  \exp\{N(\tilde{V}(\pi_1^m) -\delta )\}, \mu_N(\bar{\tau}_{\pi^m_1}) \in \pi_2^m;  \\
& \, \, \, \, \, \, \, \, E_{\mu_N(\bar{\tau}_{\pi^m_1})}(\bar{\tau}_{\pi_2^m} \leq  \exp\{N(\tilde{V}(\pi_2^m)-\delta )\},\mu_N(\bar{\tau}_{\pi^m_2}) \in \pi_3^m; \\
& \, \, \, \, \, \, \, \, \cdots  E_{\mu_N(\bar{\tau}_{\pi^m_{n-2}})}( \bar{\tau}_{\pi^m_{n-1}} \leq  \exp\{N(\tilde{V}(\pi_{n-1}^m) -\delta )\},\mu_N(\bar{\tau}_{\pi^m_{n-1}}) \in \pi_0^m)\\
& \, \, \, \, \, \, \, \, \cdots )).
\end{align*}
By Lemma~\ref{lemma:hitting_place}, there exist $\rho >0$, $\delta > 0$ and $N_0 \geq 1 $ such that each of the above probabilities is at least $\exp\{-N\varepsilon/l\}$ for sufficiently large $N$, i.e. we have
\begin{align*}
P_\nu (\hat{\tau}_{\pi^m_0} \leq n \exp\{N(V_m - \delta))\}) \geq \exp\{-N n \varepsilon/l\} \geq \exp\{-N\varepsilon\},
\end{align*}
On the other hand, if $K_j$ is such that $K_j \in \pi_0^m$, the above holds trivially. Therefore, there exist $\delta_1 > 0$ and $N_1 \geq 1$ such that	for all $\nu \in \gamma \cap M_1^N(\mathcal{Z})$ and $N \geq N_1$,  we have
\begin{align*}
P_\nu (\hat{\tau}_{\pi^m_0} \leq \exp\{N(V_m - \delta_1)\})  \geq \exp\{-N  \varepsilon\}.
\end{align*}
We now use the above bound to show~(\ref{eqn:lb_tpm_i0}). Let $T = \exp\{N(\Lambda - \delta_1)\}, T_m = \exp\{N(V_m - \delta_1)\}$ and $T_{m-1} =\exp\{N(V_{m-1}-\delta_1)\}$. Then, for any $\nu \in \gamma \cap M_1^N(\mathcal{Z})$ and $N \geq N_1$, we have
\begin{align}
P_\nu (\mu_N(T) \in \gamma_{i_0})& \geq E_\nu (\hat{\tau}_{\pi^m_0} \leq T_m; E_{\mu_N(\hat{\tau}_{\pi^m_0})}(\mu_N(T-\hat{\tau}_{\pi^m_0}) \in \gamma_{i_0})) \nonumber \\
& \geq \inf_{\substack{\nu \in [\pi_0^m]_{\rho}\cap M_1^N(\mathcal{Z}) \\  T-T_m \leq t \leq T}} P_\nu(\mu_N(t) \in \gamma_{i_0}) P_\nu (\hat{\tau}_{\pi^m_0}\leq T_m) \nonumber \\
& \geq  \inf_{\substack{\nu \in [\pi_0^m]_{\rho}\cap M_1^N(\mathcal{Z}) \\  T-T_m \leq t \leq T}} P_\nu(\mu_N(t) \in \gamma_{i_0})  \exp\{-N\varepsilon\}.
\label{eqn:temp1}
\end{align}
To get a lower bound for the above infimum, fix $\nu \in [\pi_0^m]_\rho \cap M_1^N(\mathcal{Z})$ and $T-T_{m} \leq t \leq T$. Define the stopping time $\theta \coloneqq \inf\{s> t-T_{m-1} : \mu_N(s) \in [\pi_0^m]_{\rho}\}$. Then, for a large $T^*$ (not depending on $N$) to be chosen later, we have
\begin{align}
P_\nu &(\mu_N(t) \in \gamma_{i_0}) \nonumber \\
 &\geq E_\nu (\theta \leq t-T_{m-1}+T^* , \bar{\tau}_{\pi_0^m} > T;  E_{\mu_N(\theta)}(\mu_N(t-\theta)) \in \gamma_{i_0})  \nonumber \\
& \geq P_\nu(\theta \leq t-T_{m-1}+T^* , \bar{\tau}_{\pi_0^m} > T)  \inf_{\substack{\nu^\prime \in [\pi_0^m]_{\rho}\cap M_1^N(\mathcal{Z}) \\ T_{m-1}-T^*\leq t \leq T_{m-1}}} P_{\nu^\prime} (\mu_N(t) \in \gamma_{i_0}).
\label{eqn:temp2}
\end{align}
Note that
\begin{align*}
P_\nu(\theta \leq t-T_{m-1}+T^*, \bar{\tau}_{\pi_0^m} > T)= P_\nu(\bar{\tau}_{\pi^m_0} > T) - P_\nu(\theta > t-T_{m-1}+T^* , \bar{\tau}_{\pi_0^m} > T).
\end{align*}
By Lemma~\ref{lemma:hitting_time}, since $\Lambda \leq \tilde{V}(\pi^m_0)$, we have
\begin{align*}
P_\nu(\bar{\tau}_{\pi^m_0} > T) \geq P_\nu(\bar{\tau}_{\pi^m_0} > \exp\{N(\tilde{V}(\pi^m_0) - \delta)\}) \to 1
\end{align*}
as $N \to \infty$. For the second term, note that
\begin{align*}
P_\nu& (\theta > t-T_{m-1}+T^* , \bar{\tau}_{\pi_0^m} > T) & \\
& = P_\nu (\mu_N(s) \notin [\pi_0^m]_\rho \text{ for all } t-T_{m-1}\leq s \leq t-T_{m-1}+T^*, \bar{\tau}_{\pi_0^m} > T) \\
& = P_\nu (\mu_N(s) \notin \gamma \text{ for all } t-T_{m-1}\leq s \leq t-T_{m-1}+T^*, \bar{\tau}_{\pi_0^m} > T) \\
& \leq P_\nu (\mu_N(s) \notin \gamma \text{ for all } t-T_{m-1}\leq s \leq t-T_{m-1}+T^*).
\end{align*}
The second equality follows since $\mu_N(s) \notin [\pi_0^m]_\rho$ and $\bar{\tau}_{\pi_0^m} >  T$ implies that we have exited $[\pi_0^m]_\rho$ and we have not yet entered a neighbourhood of any other attractor, which is the same as saying $\mu_N(t) \notin \gamma$ and $\bar{\tau}_{\pi_0^m} >T$. By the Markov property, the above probability equals
\begin{align*}
E_\nu  \left( E_{\mu_N(t-T_{m-1})}(\mu_N(s) \notin \gamma \text{ for all } s \in [t-T_{m-1}, t-T_{m-1}+T^*]) \right) \leq \sup_{\nu^\prime \in F} P_{\nu^\prime} (\tau_F \geq T^*),
\end{align*}
where $F = M_1(\mathcal{Z}) \setminus \gamma$. By Lemma~\ref{lemma:fw19}, $T^*$ can be chosen large enough (not depending on $N$) that the above probability is at most $1/2$. Therefore,~(\ref{eqn:temp2}) becomes
\begin{align*}
\inf_{\substack{\nu \in [\pi_0^m]_\rho \cap M_1^N(\mathcal{Z}) \\ T-T_m \leq t \leq T}} P_\nu(\mu_N(t) \in \gamma_{i_0}) \geq \frac{1}{2}   \inf_{\substack{\nu^\prime \in [\pi_0^m]_{\rho} \cap M_1^N(\mathcal{Z}) \\ T_{m-1}-T^* \leq t \leq T_{m-1}}} P_{\nu^\prime} (\mu_N(t) \in \gamma_{i_0}),
\end{align*}
and~(\ref{eqn:temp1}) becomes
\begin{align*}
P_\nu (\mu_N(T) \in \gamma_{i_0}) \geq \frac{1}{2} \exp\{-N\varepsilon\} \inf_{\substack{\nu^\prime \in [\pi_0^m]_{\rho} \cap M_1^N(\mathcal{Z}) \\  T_{m-1}-T^* \leq t \leq T_{m-1}}} P_{\nu^\prime} (\mu_N(t) \in \gamma_{i_0}),
\end{align*}
for sufficiently large $N$ and $\nu \in \gamma \cap M_1^N(\mathcal{Z})$. Repeating the above argument $m$ times, we see that there exists $N_2 \geq 1$ such that for all $\nu \in \gamma$ and $N \geq N_2$, we have
\begin{align*}
P_\nu(\mu_N(T) \in \gamma_{i_0}) & \geq \left(\frac{1}{2}\right)^m\exp\{-Nm\varepsilon\}  \inf_{\substack{\nu^\prime \in [\pi_0^1]_{\rho}\cap M_1^N(\mathcal{Z}) \\ T_0 -T^*\leq t \leq T_0}} P_{\nu^\prime} (\mu_N(t) \in \gamma_{i_0}) \\
& \geq \left(\frac{1}{2}\right)^{m}\exp\{-N(m+1)\varepsilon\} \inf_{\substack{\nu^\prime \in [K_0]_{\rho} \cap M_1^N(\mathcal{Z})\\ T_0 -T^*\leq t \leq T_0}} P_{\nu^\prime} (\mu_N(t) \in \gamma_{i_0}) \\
& \geq \left(\frac{1}{2}\right)^{m+1}\exp\{-N(m+1)\varepsilon\},
\end{align*}
where $T_0 = \exp\{N(V_0 - m\delta)\}$. Thus, we conclude that there is $N_3 \geq 1$, $\delta_3>0$ and $\rho > 0$ such that for all $\nu \in \gamma \cap M_1^N(\mathcal{Z})$ and $N \geq N_3$, we have
\begin{align*}
P_\nu(\mu_N(T) \in \gamma_{i_0})  \geq \exp\{-N(m+3)\varepsilon\},
\end{align*}
where $T = \exp\{N(\Lambda-\delta_3)\}$. This establishes~(\ref{eqn:lb_tpm_i0}) for all $\nu \in \gamma \cap M_1^N(\mathcal{Z})$. For any $\nu \in M_1^N(\mathcal{Z})\setminus \gamma$, from Lemma~\ref{lemma:fw19}, there exists $T^\prime$ large enough and $N_4 \geq N_3$ such that $P_\nu(\tau_{M_1(\mathcal{Z}) \setminus \gamma} \leq T^\prime) \leq \frac{1}{2}$ for all $N \geq N_4$. Therefore, we have
\begin{align*}
P_\nu(\mu_N(T) \in \gamma_{i_0}) & \geq E_\nu (\tau_{M_1(\mathcal{Z}) \setminus \gamma} \leq T^\prime, P_{\mu_N(\tau_F)} (\mu_N(T-T^\prime) \in \gamma_{i_0})) \\
& \geq \frac{1}{2}	\inf_{\nu^\prime  \in \gamma} P_{\nu^\prime}(\mu_N(T-T^\prime) \in \gamma_{i_0}) \\
& \geq \frac{1}{2}\exp\{-N(m+3)\varepsilon\}.
\end{align*}
Thus, we have established~(\ref{eqn:lb_tpm_i0}) for any $\nu \in M_1^N(\mathcal{Z})$.

We now turn to~(\ref{eqn:ub_tpm}). Since $\Lambda = \max\{V_k, 0 \leq k \leq m\}$, there exists a $k$ such that $V_k = \Lambda$. From the definition of $V_k$, we see that there exists $\pi^k \in L_k$ such that
\begin{align*}
\tilde{V}(\pi^k) = \Lambda, \pi^k \subset \pi_0^{k+1}, \text{ and }\pi^k \neq \pi^k_0.
\end{align*}
where $\pi_0^{k+1}$ is the $(k+1)$-cycle that contain $K_{i_0}$. Therefore, Lemma~\ref{lemma:hitting_time} implies that, for some $\beta > 0$, for some $\delta_4 < \delta_3$ and an appropriately chosen $\rho > 0$, with $T = \exp\{N(\Lambda-\delta_3)\} = \exp\{N(\tilde{V}(\pi^k) - \delta_3)\}$, we have
\begin{align*}
P_\nu(\mu_N(T) \in \gamma_{i_0})  \leq P_\nu(\bar{\tau}_{\pi^k} \leq T) \leq \exp\{-N\beta\},
\end{align*}
for any $\nu \in [\pi^k]_{\rho} \cap M_1^N(\mathcal{Z})$ and sufficiently large $N$. This completes the proof of the theorem.
\end{proof}
The above theorem immediately gives a lower bound on $P_T(\nu, \xi)$ for any $\xi$ in a small neighbourhood of  $K_{i_0}$, over time durations of order $\exp\{N(\Lambda-\delta)\}$ for some $\delta > 0$. Let us make this precise.
\begin{corollary}
Under the conditions of Theorem~\ref{thm:mixing}, for all $\nu \in M_1^N(\mathcal{Z})$, $\xi \in \gamma_{i_0} \cap M_1^N(\mathcal{Z})$ and $N$ sufficiently large, we have
\begin{align*}
P_{T_0}(\nu, \xi) \geq \exp\{-2N\varepsilon\}.
\end{align*}
\label{cor:lowerboundtransition}
\end{corollary}

\begin{proof}
Given $\varepsilon > 0$, let $\rho, N_0$ and $T_0$ be as in the statement of Theorem~\ref{thm:mixing}. Choose $t$ large enough (not depending on $N$) and $\rho^\prime < \rho$ such that for all $\rho_1 \leq \rho^\prime$ we have $S_t(\nu_1|\nu_2) \leq \varepsilon/2$ for all $\nu_1, \nu_2 \in \gamma_{i_0}$. This is possible by the joint continuity of  the rate function $S_t(\cdot|\cdot)$ and the fact that $V(\nu_1, \nu_2) = 0$ whenever $\nu_1, \nu_2 \in K_{i_0}$. Therefore, using the large deviation lower bound, there exists $N_2 \geq N_1$ such that
\begin{align*}
P_t(\nu_1, \nu_2) \geq \exp\{-N(S_t(\nu_2 | \nu_1) + \varepsilon/2)\} \geq \exp\{-N\varepsilon\},
\end{align*}
for all $\nu_1, \nu_2 \in \gamma_{i_0} \cap M_1^N(\mathcal{Z})$ and $N \geq N_2$. Therefore, by Theorem~\ref{thm:mixing}, for $\nu \in M_1^N(\mathcal{Z}), \xi \in \gamma_{i_0} \cap M_1^N(\mathcal{Z})$ and $N \geq N_2$, we have
\begin{align*}
P_{T_0}(\nu, \xi) & = \sum_{\nu_2 \in \gamma_{i_0} \cap M_1^N(\mathcal{Z})} P_{T_0-t}(\nu_1, \nu_2) P_t(\nu_2, \xi)  \\
&  \geq P_{T_0-t}(\nu_1, \gamma_{i_0}) \inf_{\nu_2 \in \gamma_{i_0} \cap M_1^N(\mathcal{Z})} P_t(\nu_2, \xi) \\
& \geq \exp\{-2N \varepsilon\}.
\end{align*}
\end{proof}
\subsubsection{Proofs of Theorem~\ref{thm:conv} and Theorem~\ref{thm:conv-converse}}
\label{subsection:thm-conv-proof}
We now prove our first main result (Theorem~\ref{thm:conv}) on the convergence of $\mu_N$ to the invariant measure and its converse Theorem~\ref{thm:conv-converse}. Theorem~\ref{thm:conv} together with Theorem~\ref{thm:conv-converse} shows that the constant $\Lambda$ is sharp (in the exponential scale) for the time required for $\mu_N$ to equilibrate.

Define $\tilde{L}_0 \coloneqq \{i \in L : W(K_i) = 0\}$, i.e, $\tilde{L}_0$ denotes the set of minimisers of the rate function $s$ (see~\ref{eqn:s_expression}). Let $B(M_1(\mathcal{Z}))$ denotes the space of bounded Borel-measurable functions on $M_1(\mathcal{Z})$.
\begin{proof}[Proof of Theorem~\ref{thm:conv}]
We follow the steps in Hwang and Sheu~{\cite[Part I, Theorem~2.5]{hwang-sheu-90}}. Let $\varepsilon > 0$, and let $T_0, \delta_0, \rho, \rho_1$ and $N_0 \geq1$ be as in the statement of Theorem~\ref{thm:mixing}.  Note that, for any $\nu \in M_1^N(\mathcal{Z})$, $\xi \notin [\tilde{L}_0]_{\rho_1}$ and for some fixed $t > 0$,
\begin{align*}
P_{T_0}(\nu,\xi)& = \sum_{\nu^\prime \in [K_{i_0}]} P_{T_0-t}(\nu, \nu^\prime) P_t(\nu^\prime, \xi)\\
& \geq \exp\{-2N\varepsilon\} \inf_{\nu^\prime \in [K_{i_0}]} P_t(\nu^\prime, \xi) \\
& \geq \exp\{-2N\varepsilon\} \exp\{-N\sup_{\nu^\prime \in [K_{i_0}]} S_t(\xi|\nu^\prime)\}
\end{align*}
where the first inequality follows from Corollary~\ref{cor:lowerboundtransition} and the second from the uniform LDP (Corollary~\ref{cor:uniform_ldp}). Hence, we can find a function $U : M_1(\mathcal{Z}) \to [0, \infty)$ such that $U(\xi) = 0$ for $\xi \in [\tilde{L}_0]_{\rho_1}$ and
\begin{align}
P_{T_0}(\nu,\xi) \geq c_N \exp\{-NU(\xi)\}
\label{eqn:temp8}
\end{align}
holds for all $\nu \in M_1^N(\mathcal{Z})$, $\xi \notin [\tilde{L}_0]_{\rho_1}$ and sufficiency large $N$; here $c_N$ is such that
\begin{align*}
\pi_N(\xi) = c_N\exp\{-NU(\xi)\}
\end{align*}
is a probability measure on $M_1^N(\mathcal{Z})$. Define $Q_{T_0}(\nu ,\cdot) \coloneqq P_{T_0}(\nu, \cdot)/ \pi(\cdot)$. We have, for any $\nu_1, \nu_2 \in M_1^N(\mathcal{Z})$ and sufficiently large $N$,
\begin{align*}
E_{\nu_1}&(f(\mu_N(T_0))) - E_{\nu_2}(f(\mu_N(T_0)))  \\
& = \sum_{\xi \in M_1^N(\mathcal{Z})}P_{T_0}(\nu_1,\xi) f(\xi) - \sum_{\xi \in M_1^N(\mathcal{Z})}P_{T_0}(\nu_2,\xi) f(\xi) \\
& = \sum_{\xi \in M_1^N(\mathcal{Z})}Q_{T_0}(\nu_1,\xi) f(\xi)  \pi_N(\xi)- \sum_{\xi \in M_1^N(\mathcal{Z})}Q_{T_0}(\nu_2,\xi) f(\xi) \pi_N(\xi) \\
& = \sum_{\xi \in M_1^N(\mathcal{Z})}(Q_{T_0}(\nu_1,\xi)- \exp\{-2N\varepsilon\})f(\xi) \pi_N(\xi) \\
& \, \, \, \, \, \, \, \, -\sum_{\xi \in M_1^N(\mathcal{Z})}(Q_{T_0}(\nu_2,\xi) - \exp\{-2N\varepsilon\}) f(\xi) \pi_N(\xi) \\
& \leq (1-\exp\{-2N\varepsilon\}) (\sup_\xi f(\xi) - \inf_\xi f(\xi)),
\end{align*}
where the last inequality follows from~(\ref{eqn:temp8}) and the fact that $Q_{T_0}(\cdot, \cdot) \geq 1$. Therefore, we have that
\begin{align*}
\sup_{\nu_1, \nu_2} | E_{\nu_1}&(f(\mu_N(T_0))) - E_{\nu_2}(f(\mu_N(T_0)))  | \leq (1-\exp\{-2N\varepsilon\}) \| f\|_\infty.
\end{align*}
Continuing this procedure $k$ times, and by using the Markov property, we get
\begin{align*}
\sup_{\nu_1, \nu_2} |E_{\nu_1}&(f(\mu_N(kT_0))) - E_{\nu_2}(f(\mu_N(kT_0))) |  \leq (1-\exp\{-2N\varepsilon\})^k \| f\|_\infty,
\end{align*}
and hence, we have
\begin{align*}
\sup_{\nu} | E_{\nu}&(f(\mu_N(kT_0))) - \langle f, \wp_N \rangle | \leq (1-\exp\{-2N\varepsilon\})^k \| f\|_\infty.
\end{align*}
Choose $k = \exp\{N(\delta_0+\delta)\}$, then we have $kT_0 = \exp\{N(\Lambda + \delta)\}$ and the above becomes
\begin{align*}
\sup_{\nu} | E_{\nu}&(f(\mu_N(kT_0))) - \langle f, \wp_N \rangle | \leq \exp\{-\exp(N(-2\varepsilon+\delta_0 + \delta))\}.
\end{align*}
We can choose $\varepsilon$ small enough such that the quantity $-2\varepsilon+ \delta > 0$, and hence for some $\varepsilon^\prime > 0$, we have
\begin{align*}
\sup_{\nu} | E_{\nu}&(f(\mu_N(T))) - \langle f, \wp_N \rangle | \leq \exp\{-\exp(N\varepsilon^\prime)\},
\end{align*}
for sufficiently large $N$, where $T = \exp\{N(\Lambda+\delta)\}$. This establishes the result.
\end{proof}
\begin{proof}[Proof of Theorem~\ref{thm:conv-converse}]
This is a direct consequence of~\eqref{eqn:ub_tpm} established in Theorem~\ref{thm:mixing}.
\end{proof}
\section{Asymptotics of the second largest eigenvalue for reversible processes}
\label{section:eval_problem}
In this section, our goal is to understand the convergence rate of $\mu_N$ to its invariant measure for a fixed $N$. For this purpose, we shall assume that the Markov process $\mu_N$ is reversible. That is, the operator $L^N$ is self-adjoint in $L^2(\wp_N)$ and it admits a spectral expansion; let $0 =\lambda_1^N > -\lambda_2^N \geq -\lambda_3^N \ldots$ denote its eigenvalues in the decreasing order, and let $u_1^N \equiv 1, u_2^N, u_3^N,\ldots$ denote their corresponding eigenfunctions. The spectral expansion enables us to write, for any $f \in B(M_1(\mathcal{Z}))$,
\begin{align}
E_\nu f(\mu_N(t)) = \langle f, \wp_N\rangle + \sum_{k \geq 2} e^{-t\lambda_k^N} ( f, u_k^N ) u_k^N(\nu),
\label{eqn:spectral_expansion}
\end{align}
where $( \cdot, \cdot )$ denotes the inner product in $L^2(\wp_N)$. Therefore, the convergence rate of $E_\nu f(\mu_N(t))$ to its stationary value $\langle f, \wp_N \rangle$ is determined by the leading term in the above sum, which is the second largest eigenvalue $\lambda_2^N$. Hence, to understand convergence of $\mu_N$ to its invariant measure, we study the  asymptotics of the second largest eigenvalue $\lambda_2^N$.

We first need the following lemma that estimates the probability that the process $\mu_N$ is outside a small neighbourhood of the set $\cup_{i=1}^l K_i$. This can be shown using Theorem~\ref{thm:conv} with deals with the convergence to the invariant measure and Theorem~\ref{thm:invariant-measure-gase} which addresses large deviations of the invariant measure $\{\wp_N\}_{N \geq 1}$.
\begin{lemma}
Fix $\rho_1 >0$ and let $B$ be the $\rho_1$-neighbourhood of $\cup_{i \in L} K_i$. Given $\varepsilon > 0$, there exist $\delta >0$ and $N_0 \geq 1$ such that for each $\nu \in M_1^N(\mathcal{Z})$ and $N \geq N_0$, we have
\begin{align*}
P_\nu \left(\mu_N(T) \in M_1^N(\mathcal{Z}) \setminus B\right) \leq \exp\{-N\delta\},
\end{align*}
where $T = \exp\{N(\Lambda+ \varepsilon)\}$.
\end{lemma}
We are now ready to prove our next main result (Theorem~\ref{thm:eval_problem}) on the asymptotics of the second largest eigenvalue $\lambda_2^N$.
\begin{proof}[Proof of Theorem~\ref{thm:eval_problem}]
(Lower bound): Suppose that there exists a subsequence $\{N_k\}_{k \geq 1}$ such that
\begin{align}
\log \lambda_2^{N_k} < -N_k(\Lambda + \varepsilon)
\label{eqn:temp-assumption}
\end{align}
for some $\varepsilon > 0$. We will show that this contradicts $\int (u_2^{N_k}(\nu))^2 \wp_N(d\nu) =1 $ for sufficiently large $k$. Fix $\rho >0 $ and define $B \coloneqq \cup_{i=1}^l [K_i]_{\rho}$. Then, using the lower semicontinuity of the rate function $S_t(\cdot|\cdot)$ and Corollary~\ref{cor:uniform_ldp} on uniform LDP, we see that for sufficiently large $t$, there exists $\delta_1 >0$ such that $\inf\{S_t(\xi|\nu):\xi, \nu \in B^c\} = \delta_1 > 0$. Therefore, for any $\nu \in B^c \cap M_1^N(\mathcal{Z})$ and any $\delta_2>0$, there exists $N_0 \geq 1$ such that for all $N \geq N_0$,
\begin{align*}
P_\nu(\mu_N(t) = \nu) \leq \exp\{-N(S_t(\nu|\nu) - \delta_2)\} \leq \exp\{-N(\delta_1 + \delta_2)\}.
\end{align*}
On the other hand, (\ref{eqn:spectral_expansion}) implies that,
\begin{align*}
P_\nu(\mu_N(t) = \nu) & = E_\nu(1_{\nu}(\mu_N(t))) \\
&\geq e^{-\lambda_2^Nt} (u_2^N(\nu))^2 \wp_N(\nu),
\end{align*}
so that
\begin{align}
\int_{B^c} |u_2^N|^2 \wp_N (d \nu) \leq \exp\{-N(\delta_1 + \delta_2)\}
\label{eqn:temp11}
\end{align}
for all $N \geq N_0$. To bound the integral over $B$, by Theorem~\ref{thm:conv}, with $T = \exp\{N(\Lambda+\varepsilon/2)\}$, there exist $\delta_3>0$ and $N_1 \geq N_0$ such that for all $N \geq N_1$,
\begin{align*}
\left|E_\nu f(\mu_N(T)) - \langle f, \wp_N \rangle \right| \leq \|f\|_\infty \exp\{-\exp(N\delta_3)\},
\end{align*}
for any $f \in B(M_1(\mathcal{Z}))$. On the other hand, from~(\ref{eqn:spectral_expansion}), for any $\nu \in B \cap M_1^N(\mathcal{Z})$, with $f = 1_\nu$, we have

\begin{align*}
\left| E_\nu f(\mu_N(T)) - \langle f, \wp_N \rangle \right| & = \sum_{i \geq 2} \exp\{-\lambda_i^NT\} \langle f, \wp_N(\nu) \rangle u_i^{N}(\nu) \\
& \geq \exp\{-\lambda_2^{N}T\} (u_2^{N}(\nu))^2 \wp_N(\nu),
\end{align*}
so that, by our assumption~(\ref{eqn:temp-assumption}), there exists a $k_0 \geq 1$ such that
\begin{align*}
u_2^{N_k}(\nu))^2 \wp_{N_k}(\nu) & \leq \exp\{\lambda_2^{N_k}T\}  \exp\{-\exp(N_k\delta_3)\} \\
& \leq  \exp\{2\exp(-N_k(\Lambda+\varepsilon)) \exp(N_k(\Lambda+\varepsilon/2))\} \exp\{-N_k\delta_3\}
\end{align*}
for all $k \geq k_0$. Since $|M_1^{N_k}(\mathcal{Z})| \leq (N_k+1)^{|\mathcal{Z}|}$ for all $k$, the above implies that, for some $\delta_4 > 0$,
\begin{align}
\int_B (u_2^{N_k}(\nu))^2 \wp_{N_k}(d\nu) \leq \exp\{-N_k\delta_4\}
\label{eqn:temp12}
\end{align}
for all $k \geq k_0$. Therefore,~(\ref{eqn:temp11}) and~(\ref{eqn:temp12}) implies that, for some $\delta > 0$,
\begin{align*}
\int_{M_1(\mathcal{Z})}  (u_2^{N_k}(\nu))^2 \wp_{N_k}(d\nu) \leq \exp\{-N_k\delta\}
\end{align*}
for all sufficiently large $k$, which is a contradiction to $\int  (u_2^{N_k}(\nu))^2 \wp_{N_k} (d\nu) =1$ for all sufficiently large $k$.

(Upper bound):
Suppose that there exists a subsequence $\{N_k\}_{k \geq 1}$ such that $\log \lambda_2^N > N_k(-\Lambda + \varepsilon)$ for some $\varepsilon>0$. Let $\nu_0, \delta_0<\varepsilon/2, \rho, N_0$ be as in Theorem~\ref{thm:mixing}. Then, with $f(\nu) =1_{[K_{i_0}]_{\rho/2}}(\nu)$ and $T = \exp\{N(\Lambda-\delta_0/2)\}$,~(\ref{eqn:ub_tpm}) implies that
\begin{align*}
E_\nu f(\mu_N(T)) = P_\nu(\mu_N(T) \in [K_{i_0}]_{\rho/2}) \leq \exp\{-N\beta\}
\end{align*}
for all $N \geq N_0$ and $\nu \in  [\nu_0]_{\rho/2 } \cap M_1^N(\mathcal{Z})$. Also, by Theorem~\ref{thm:invariant-measure-gase}, for any $\delta >0$, there exists $N_1 \geq N_0$ such that for all $N \geq N_1$, we have
\begin{align*}
\langle f, \wp_N \rangle = \wp_N([K_{i_0}]_{\rho/2}) \geq \exp\{-N\delta\}.
\end{align*}
This is possible since $\inf_{\xi \in [K_{i_0}]_{\rho/2}}s(\xi) =0$. Therefore, for all $N \geq N_1$,
\begin{align*}
\int_{M_1(\mathcal{Z})} |E_\nu(f(\mu_N(T))) & - \langle f, \wp_N \rangle |^2 \wp_N(d\nu) \\
&\geq \int_{[\nu_0]_{\rho/2}} \left|E_\nu(f(\mu_N(T))) - \langle f, \wp_N \rangle \right|^2 \wp_N(d\nu) \\
& \geq \wp_N([\nu_0]_{\rho/2}) (\exp\{-N\beta\} - \exp\{-N\delta\}) \\
&\geq \wp_N([\nu_0]_{\rho/2}) \exp\{-N\delta_1\}, \text{ for some } \delta_1 > 0 \\
&\geq \exp\{-N\delta_2\}, \text{ for some } \delta_2 > 0,
\end{align*}
where the last inequality follows by Theorem~\ref{thm:invariant-measure-gase}. On the other hand, for any function $f$ with $\int |f|^2 d\wp_N \leq 1$, we have
\begin{align*}
\int_{M_1(\mathcal{Z})} |E_\nu(f(\mu_N(T))) & - \langle f, \wp_N \rangle |^2 \wp_N(d\nu)  \\
& =  \int_{M_1(\mathcal{Z})} \sum_{k\geq 2} e^{-2\lambda_k^NT} \langle f, u_2^N\rangle u_2^N(\nu)^2 \wp_N(d\nu) \\
&\leq \exp\{-2\lambda_2^NT\} \int_{M_1(\mathcal{Z})} |f|^2 d\wp_N  \\
&\leq \exp\{-2\lambda_2^NT\}.
\end{align*}
Therefore, we have $ \exp\{-2\lambda_2^{N}T\} \geq \exp\{-N\delta_2\}$ whenever $N \geq N_1$. By our assumption, we see that
\begin{align*}
 \exp\{-2\exp(-N_k(\Lambda-\varepsilon)) \exp(N_k(\Lambda-\delta_0))\} \geq \exp\{-N_k\delta_1\}
\end{align*}
for sufficiently large $k$, which is a contradiction since $\delta_0 < \varepsilon/2$.
\end{proof}
Using the above theorem, we see that, if $\Lambda>0$, then as $N$ becomes large, it takes longer for the process $\mu_N$ to be close to its invariant measure. This particularly means that metastable states reduce the rates of convergence of $\mu_N$ to its invariant measure. On the other hand, if there is a unique global attractor of the limiting McKean-Vlasov equation~(\ref{eqn:MVE}), then we see that $\Lambda = 0$, and convergence rate of $\mu_N$ to its invariant measure does not suffer from such a slowing down phenomenon.
\begin{remark}
Note that the spectral expansion in~(\ref{eqn:spectral_expansion}) is crucial in the proof of Theorem~\ref{thm:eval_problem} to be able to use the results on large time behaviour of $\mu_N$ established in Section~\ref{section:large_time_behaviour} to obtain the asymptotics of $\lambda^2_N$. The main purpose of Theorem~\ref{thm:eval_problem} is to demonstrate that, in the reversible case, the asymptotics of $\lambda^2_N$ can be easily obtained as an application of the study of the large time behaviour of $\mu_N$. Even in the non-reversible case, one can obtain asymptotics of the real part of $\lambda^2_N$ via other approaches; see, for example,~\cite{wentzell-75}, where the author obtains the asymptotics of the real part of the second largest eigenvalue of the generator corresponding to a small noise diffusion process via examining eigenvalues of a discrete time chain (with transition probabilities of the form appearing in~(\ref{eqn:tpm_zn})) and transferring them to the operator.
\end{remark}
\begin{remark}One can construct examples where $\mu_N$ is reversible with respect to $\wp_N$. For instance, in the non-interacting case (i.e. when, for each $(z,z^\prime) \in \mathcal{E}$, $\lambda_{z,z^\prime}(\cdot)$ is a constant function, which we denote by $\lambda_{z,z^\prime}$) where the Markov process on $\mathcal{Z}$ with generator
\begin{align*}
f \mapsto \sum_{z^\prime: (z,z^\prime) \in \mathcal{E}} (f(z^\prime) - f(z)) \lambda_{z,z^\prime}, z \in \mathcal{Z}
\end{align*}
is reversible with respect to its invariant measure (i.e. when the Markov process corresponding to a single particle's evolution on $\mathcal{Z}$ is reversible with respect to its invariant measures) results in a reversible empirical measure process $\mu_N$. However, the authors are not aware of a general condition (in terms of the transition rates $\lambda_{z,z^\prime}(\cdot), (z,z^\prime) \in \mathcal{E}$) that characterises reversibility of $\mu_N$.
\end{remark}
\section{Convergence to a global minimum via controlled addition of particles}
\label{section:conv_global_minimum}

In this section, our goal is to increase the number of particles $N$ over time so as to obtain, with high probability, convergence of the empirical measure process to a global minimum of the rate function $s$ that governs the LDP for the sequence of invariant measure $\{\wp_N\}_{N\geq 1}$.

Fix $c>0$. Let $N_0 = \min\{n \in \mathbb{N} : \exp\{nc\} - 2 \geq 0 \}$, $t_{N_0} = 0$, and for each $N > N_0$, let $t_N = \exp\{Nc\}-2$. For each $N \geq N_0$ define the generator $L^N_t$ acting on bounded measurable functions on $M_1(\mathcal{Z})$ by
\begin{align*}
L_t^Nf (\xi) \coloneqq \sum_{(z,z^\prime) \in E} N_t \xi(z) \lambda	_{z,z^\prime}(\xi) \left[ f\left(\xi+\frac{e_{z^\prime}}{N_t}- \frac{e_z}{N_t}\right) - f(\xi) \right], \, t\in [t_{N}, t_{N+1}).
\end{align*}
 where $N_t = N$ for $t \in [t_N, t_{N+1})$. Let $z_0 \in \mathcal{Z}$ be a fixed state and let $\nu \in M_1^{N_0}(\mathcal{Z})$. We say that a probability measure $P_{0,\nu}$ on $D([0,\infty),M_1(\mathcal{Z}))$ is a solution to the martingale problem for $\{L^N\}_{N \geq N_0}$ with initial condition $\nu$ if $P_{0,\nu}(\bar{\mu}: \bar{\mu}(0) = \nu)=1$, for each $N \geq N_0$, the restriction of $P_{0,\nu}$ on $D([t_N, t_{N+1}), M_1^{N}(\mathcal{Z}))$ is a solution to the $D([t_N, t_{N+1}), M_1^{N}(\mathcal{Z}))$-valued martingale problem for $L^N$,  and
\begin{align*}
P_{0,\nu}\left(\bar{\mu}: \bar{\mu}(t_{N+1}) = \frac{N}{1+N}\bar{\mu}(t_{N+1}^-) + \frac{1}{N+1} \delta_{z_0}\right)  =1.
\end{align*}
Again, by the boundedness assumption on transition rates~\ref{assm:a2}, for each $\nu \in M_1^{N_0}(\mathcal{Z})$, there exists a unique probability measure $P_{0,\nu}$ that solves the martingale problem for $\{L^N\}_{N \geq N_0}$ with initial condition $\nu$. Let $\bar{\mu}$ be the process on $D([0,\infty),M_1(\mathcal{Z}))$ whose law is $P_{0,\nu}$. To describe the process in words, we start with $N_0$ particles and follow the mean-field interaction described in Section~\ref{section:introduction}, except that at each time instant $t_N, N > N_0$, we add a new particle whose state is set to $z_0$.

We anticipate that if $c$ is small then $N_t$ is so large  that the fluid limit kicks in too quickly over time and the process $\bar{\mu}$  converges (over time) to a local minimum of $s$ with positive probability depending on the initial condition $\bar{\mu}(0)$. When $c$ is sufficiently large, we anticipate that there is enough time for exploration and therefore we will converge to a global minimum of $s$. Recall that the set of global minimisers of $s$ is denoted by $\tilde{L}_0$. Our interest in this section is in finding a constant $c^*$ such that for all $c >  c^*$ and $\nu \in M_1^{N_0}(\mathcal{Z})$, we have,
\begin{align}
P_{0,\nu} (\bar{\mu}(t) \text{ lies in a neighbourhood of }\tilde{L}_0) \to  1
\label{eqn:conv_globalmin}
\end{align}
as $t \to \infty$.

We use the results in the previous sections to identify the constant $c^*$. Since $N_t \to \infty$ as $t \to \infty$, for a fixed $T > 0$ and large enough $t$, the large deviation properties of the process $\{\bar{\mu}(s),  t \leq s \leq t+T \}$ from the limiting dynamics~(\ref{eqn:MVE}) starting at an arbitrary $\bar{\mu}(t)$ can be obtained similar to the LDP of the process $\mu_N$ studied in Theorem~\ref{thm:finite-duration-ldp} and Corollary~\ref{cor:uniform_ldp}.  Therefore, the results in the previous sections on the large time behaviour for the process $\{\mu_N(t), t\geq 0\}$ are also valid for $\{\bar{\mu}(t) , t \geq 0\}$ when time $t$ is large enough; we make these precise now.

\begin{lemma}[see Lemma~\ref{lemma:hitting_place}]
Let $\pi_1^k$ and $\pi_2^k$ be $k$-cycles and suppose that $\pi_1^k \to \pi_2^k$ and $\tilde{V}(\pi_1^k)/c < 1$. Then, given $\varepsilon > 0$, there exist $\delta>0$ and  $\rho>0$ such that for all $\rho_1 < \rho$, there is $t^* >0$ such that
\begin{align*}
P_{t,\nu}(\bar{\tau}_{\pi^k_1} \leq t + t^{(\tilde{V}(\pi_1^k) -\delta)/c}, \bar{\mu}(\bar{\tau}_{\pi_1^k}) \in \gamma_{\pi_2^k}) \geq t^{-\varepsilon/c}
\end{align*}
holds uniformly for all $\nu \in [\pi_1^k]_{\rho_1} \cap M_1^{N_{t}}(\mathcal{Z})$ and $t \geq t^*$.
\label{lemma:hs22}
\end{lemma}
\begin{remark}The condition $\tilde{V}(\pi_1^k) /c <1$ in the above lemma ensures that during the time duration $[t, t^{\tilde{V}(\pi_1^k)/c}]$, for large enough $t$, the number of particles does not change so that Lemma~\ref{lemma:hitting_place} for the process $\mu_N$ is applicable for the process $\bar{\mu}$.
\end{remark}
\begin{lemma}[see Lemma~\ref{lemma:hitting_time}]
Let $\pi^k$ be a $k$-cycle and suppose that $\tilde{V}(\pi^k)/c < 1$. Then, given $\delta > 0$ such that $(\tilde{V}(\pi^k)+\delta)/c < 1$, there exist $\varepsilon > 0$ and $\rho>0$ such that for all $\rho_1 < \rho$, there is $t^* > 0$ such that
\begin{align*}
P_{t,\nu}(\bar{\tau}_{\pi^k} & <  t + t^{(\tilde{V}(\pi^k) -\delta)/c}) \leq t^{-\varepsilon/c}, \text{ and }\\
P_{t,\nu}(\bar{\tau}_{\pi^k}& > t + t^{(\tilde{V}(\pi^k) +\delta)/c}) \leq t^{-\varepsilon/c}
\end{align*}
holds uniformly for all $\nu \in [\pi^k]_{\rho_1} \cap M_1^{N_{t}}(\mathcal{Z})$  and $t \geq t^*$.
\label{lemma:hs23}
\end{lemma}

\begin{lemma}[see Lemma~\ref{lemma:exit_time_vhat}]
Let $\pi^k$ be a $k$-cycle and suppose that $\hat{V}(\pi^k)/c < 1$. Given $\varepsilon>0$, there exist $\delta  \in (0, c-\hat{V}(\pi^k))$ and $\rho>0$ such that for all $\rho_1 \leq \rho$, there is $t^* >0$ such that
\begin{align*}
P_{t,\nu}(\bar{\tau}_{\pi^k}\leq t + t^{(\hat{V}(\pi^k)+\delta)/c}) \leq t^{-(\tilde{V}(\pi^k) - \hat{V}(\pi^k)-\varepsilon)/c}
\end{align*}
holds uniformly for all $\nu \in [\pi^k]_{\rho_1} \cap M_1^{N_{t}}(\mathcal{Z})$  and $t \geq t^*$.
\label{lemma:hs24}
\end{lemma}

Recall the definition of the sets $L$ and $C$ from Section~\ref{section:large_time_behaviour}.
\begin{lemma}[see Lemma~\ref{lemma:fw19}]
\label{lemma:conv-fw19}
Given $\rho_0 > 0$ and $\rho_1 <\rho_0$ and their associated sets $L$ and $C$, given $v>0$, there exist $T^* > 0$ and $t^* >0 $ such that
\begin{align*}
P_{t,\nu}(\hat{\tau}_L \geq t + T^*) \leq t^{-v/c}
\end{align*}
holds uniformly for all $\nu \in C \cap M_1^{N_{t}}(\mathcal{Z})$ and $t \geq t^*$.
\end{lemma}

To answer the question on the convergence of $\bar{\mu}$ to a global minimum of $s$, we define the following quantities, analogous to what is done in Hwang and Sheu~\cite{hwang-sheu-90}. Let $m$ be such that $L_{m+1}$ is a singleton (denote it by $\{\pi^{m+1}\}$). Define
\begin{align*}
A_m \coloneqq \{\pi^m \in L_m:\tilde{V}(\pi^m) = \hat{V}(\pi^{m+1})\}.
\end{align*}
Inductively define, for each $\pi^{k+1} \in L_{k+1}$,
\begin{align*}
A_{k}(\pi^{k+1}) \coloneqq \{\pi^k \in \pi^{k+1} : \tilde{V}(\pi^k) = \hat{V}(\pi^{k+1})\},
\end{align*}
and for each $k \geq 1$, define
\begin{align*}
A_k \coloneqq \bigcup_{\pi^{k+1} \in A_{k+1}} A_k(\pi^{k+1}).
\end{align*}
Also, for each  $\pi^k \in L_k$, define
\begin{align*}
c_{k-1}(\pi^k)  \coloneqq \left\{\begin{array}{ll}
0,  \text{ if } \{ \pi^{k-1} \in \pi^k : \pi^{k-1} \notin A_{k-1}(\pi^k) \} = \emptyset, \\
\max\{\tilde{V}(\pi^{k-1}):\pi^{k-1} \notin A_{k-1}(\pi^k), \pi^{k-1} \in \pi^k\}, \text{  otherwise},
\end{array}
\right.
\end{align*}
and for each $k \geq 1$, define
\begin{align*}
c_{k-1} \coloneqq \max\{c_{k-1}(\pi^k), :\pi^k \in A_k\}.
\end{align*}
Finally, define
\begin{align*}
c^* \coloneqq \max\{c_k, 0 \leq k \leq m\}.
\end{align*}
Similar to~\cite[Lemma~A.11,~Appendix]{hwang-sheu-90}, we can show that $A_0 = \tilde{L}_0$, the set of minimisers of the rate function $s$ that governs the LDP for the invariant measure $\{\wp_N\}_{N\geq 1}$. We now prove Theorem~\ref{thm:conv-globalmin} on convergence of $\bar{\mu}$ to the set of global minimisers.
\begin{proof}[Proof of Theorem~\ref{thm:conv-globalmin}]
It suffices to show that, for any $\delta > 0$ with $(c^*+\delta)/c<1$, there exist $\varepsilon > 0$, $\rho_1>0$ and $t^* > 0$ such that
\begin{align*}
P_{t, \nu} (\bar{\mu}(t + t^{(c^*+\delta)/c}) \in [\tilde{L}_0]_{\rho_1}) \geq 1- t^{-\varepsilon/c}
\end{align*}
for all $t > t^*$ and $\nu \in M_1^{N_{t}}(\mathcal{Z})$. Define the stopping time
\begin{align*}
\theta \coloneqq \inf\{s>t: \bar{\mu}(s) \in [L]_{\rho_1}\}.
\end{align*}
By Lemma~\ref{lemma:conv-fw19}, for any $M >0$, there exists $T^*>0$ such that for all $\nu \in M_1^{N_0}(\mathcal{Z})$ and large enough $t$, we have
\begin{align*}
P_{t, \nu} (\theta > t + T^*) \leq t^{-M/c}.
\end{align*}
By the strong Markov property, we have
\begin{align}
P_{t, \nu}&( \bar{\mu}(t + t^{(c^*+\delta)/c})\in [\tilde{L}_0]_{\rho_1}) \nonumber \\
& \geq E_{t,\nu} (\theta \leq t + T^*; E_{\theta, \bar{\mu}(\theta)} (\bar{\mu}(t+t^{(c^*+\delta)/c}) \in [\tilde{L}_0]_{\rho_1})) \nonumber \\
& \geq \inf_{\substack{t \leq t_1 \leq t+T^* \\ \nu_1 \in [L]_{\rho_1}}} P_{t_1, \nu_1} (\bar{\mu}(t+t^{(c^*+\delta)/c}) \in [\tilde{L}_0]_{\rho_1}) (1-t^{-M/c}).
\label{eqn:temp7}
\end{align}
To bound the first term above, fix a $t_1$ such that $t \leq t_1 \leq t+T^*$ and $\nu_1 \in [L]_{\rho_1}$. Define the stopping time $\theta_m \coloneqq \inf \{t>t_1: \bar{\mu}(t) \in [A_m]_{\rho_1}\}$. We have
\begin{align}
P_{t_1, \nu_1} &(\bar{\mu}(t+t^{(c^*+\delta)/c}) \in [\tilde{L}_0]_{\rho_1}) \nonumber \\
& \geq E_{t_1, \nu_1} (\theta_m < t+t^{(c^*+\delta/2)/c}; E_{\theta_m, \bar{\mu}(\theta_m)}(\bar{\mu}(t+t^{(c^*+\delta)/c}) \in [\tilde{L}_0]_{\rho_1})) \nonumber \\
& \geq \inf_{t \leq t_2 \leq t+t^{(c^*+\delta/2)/c}, \nu_2 \in [A_m]_{\rho_1}} P_{t_2, \nu_2}(\bar{\mu}(t+t^{(c^*+\delta)/c}) \in [\tilde{L}_0]_{\rho_1})  \nonumber \\
& \, \, \, \, \, \, \, \, \, \, \times P_{t_1, \nu_1}(\theta_m  \leq t+t^{(c^*+\delta/2)/c}). \label{eqn:temp6}
\end{align}
We first bound the second term $P_{t_1, \nu_1}(\theta_m  \leq t+t^{(c^*+\delta/2)/c})$. Note that, by Lemma~\ref{lemma:hs22}, for any $M_1>0$, there exists $\delta_1>0$ such that
\begin{align*}
P_{t_1, \nu_1}(\theta_m  > t_1+t_1^{(c_m-\delta_1)/c}) \leq 1-t_1^{-M_1/c}
\end{align*}
for sufficiently large $t$. Let $T_1 = t_1 + t_1^{(c_m-\delta_1)/c}$, and define the stopping time $\hat{\theta} \coloneqq \inf\{t>T_1: \bar{\mu}(t) \in [L]_{\rho_1} \}$. Again, by Lemma~\ref{lemma:conv-fw19}, there exists a large enough $T^*$ such that $P_{T_1, \nu}(\hat{\theta} > T_1+T^*) \leq T_1^{-M/c}$ for all $\nu \in M_1^{N_{T_1}}(\mathcal{Z})$. Therefore, using the strong Markov property, we have
\begin{align}
P_{t_1, \nu_1}&(\theta_m  > t+t^{(c^*+\delta/2)/c}) \nonumber \\
& \leq E_{t_1, \nu_1}(\theta_m \geq \hat{\theta}, \hat{\theta}< T_1+T^*; E_{\hat{\theta}, \bar{\mu}(\hat{\theta})}(\theta_m >t+t^{(c^*+\delta/2)/c})) \nonumber \\
& \, \, \, \, \, \, \, \, + P_{t_1, \nu_1}(\hat{\theta} > T_1+T^*) \nonumber \\
& \leq P_{t_1, \nu_1}(\theta_m > T_1) \sup_{\substack{T_1 \leq t \leq T_1 + T^* \\  \nu \in [L]_{\rho_1}}} P_{t,\nu}(\theta_m > t+t^{(c^*+\delta/2)/c}) +t_1^{-M/c} \nonumber\\
& \leq (1-t_1^{-M_1/c} ) \sup_{\substack{T_1 \leq t \leq T_1 + T^* \\ \nu \in [L]_{\rho_1}}} P_{t,\nu}(\theta_m > t+t^{(c^*+\delta/2)/c}) +t_1^{-M/c}.  \label{eqn:temp5}
\end{align}
We now focus on $P_{t,\nu}(\theta_m > t+t^{(c^*+\delta/2)/c})$ for a fixed $t \in [T_1, T_1+T^*]$ and $\nu \in [L]_{\rho_1}$, and repeat the above steps; this will introduce a multiplication factor of $(1-T_1^{-M_1/c})$ along with
\begin{align*}
\sup_{\substack{T_2 \leq t \leq T_2 + T^* \\ \nu \in [L]_{\rho_1}}} P_{t,\nu}(\theta_m > t+t^{(c^*+\delta/2)/c}),
\end{align*}
where $T_2 = T_1 + T_1^{(c_m-\delta_1)/c}$, in the first term in~(\ref{eqn:temp5}), and an addition of $t_1^{-M/c}$ in the second term. Therefore, repeating the above steps $r \sim t_1^{\delta/2c}$ times, we get
\begin{align*}
P_{t_1, \nu_1}(\theta_m  > t+t^{(c^*+\delta/2)/c}) \leq  \prod_{n=0}^r (1-T_n^{*{-M_1/c}}) + rt_1^{-M/c},
\end{align*}
where $T_0^* = t_1$, and
\begin{align*}
T_{n+1}^* = T_n^* + T_n^{*{(c_m-\delta_1)/c}} + T^*.
\end{align*}
Note that,
\begin{align}
 \prod_{n=0}^r (1-T_n^{*{-M_1/c}}) & \leq  \exp\left\{-\sum_{n=0}^r T_n^{*{-M_1/c}} \right\} \nonumber \\
 & = \exp\left\{ -\sum_{n=0}^r T_n^{*{-M_1/c - (c_m-\delta_1)/c}}(T_{n+1}^* - T_n^*) \right\} \nonumber \\
 & \leq  \exp\left\{- \int_{T_0^*}^{T_r^*} u^{-(M_1/c) - (c_m-\delta_1)/c} du\right\} \nonumber \\
& = \exp \left\{ - \left( T_r^{*{1-(c_m+M_1-\delta_1)/c}} -  t_1^{1-(c_m+M_1-\delta_1)/c} \right)\right\}. \label{eqn:temp14}
\end{align}
Since $T_n \geq t_1$ for all $n \geq 1$, we see that $T_r^* \geq t_1 + r t_1^{(c_m-\delta_1)/c} \sim t_1 + t_1^{(c_m-\delta_1+\delta/2)/c}$. Therefore,
\begin{align*}
-  & \left(T_r^{*{1-(c_m+M_1-\delta_1)/c}} -  t_1^{1-(c_m+M_1-\delta_1)/c} \right) \\
& \leq -\left( (t_1+t_1^{(c_m-\delta_1+\delta/2)/c})^{1-(c_m+M_1-\delta_1)/c} - t_1^{1-(c_m+M_1-\delta_1)/c}\right) \\
& \leq -\left( t_1^{1-(c_m+M_1-\delta_1)/c} \left( 1+t_1^{(c_m-\delta_1+\delta/2)/c-1}\right)^{1-(c_m+M_1-\delta_1)/c} -1\right) \\
& \leq -c^\prime \left(t_1^{1-(c_m+M_1-\delta_1)/c} t_1^{(c_m-\delta_1+\delta/2)/c-1} \right)  \\
& = -c^\prime t_1^{(\delta/2 - M_1)/c},
\end{align*}
for some constant $c^\prime  >0$ and large enough $t_1$. Hence,~(\ref{eqn:temp14}) becomes
\begin{align*}
 \prod_{n=0}^r (1-T_n^{*^{-M_1/c}}) \leq \exp\{-c^\prime t_1^{(\delta/2-M_1)/c}\}.
\end{align*}
We choose $M_1 = \delta/4$; the above and~(\ref{eqn:temp5}) then implies
\begin{align*}
P_{t_1, \nu_1}&(\theta_m  > t+t^{(c^*+\delta/2)/c}) \leq \exp\{-c^\prime t_1^{\delta/4c}\} + t_1^{-(M-\delta/2)/c},
\end{align*}
and this implies that, for any $M^\prime > 0$,
\begin{align}
P_{t_1,\nu_1} (\theta_m > t+t^{(c^*+\delta/2)/c}) \leq t^{-M^\prime/c}
\label{eqn:temp13}
\end{align}
for sufficiently large $t$, $t \leq t_1 \leq t+T^*$ and for all $\nu \in [L]_{\rho_1}$.

We now bound the first term in~(\ref{eqn:temp6}), $P_{t_2,\nu_2}(\bar{\mu}(t+t^{(c^*+\delta)/c}) \in [\tilde{L}_0]_{\rho_1})$ where $t \leq t_2 \leq t + t^{(c^*+\delta/2)/c}$ and $\nu_2 \in [A_m]_{\rho_1}$.  Let $\pi_0^m \in A_m$ be the $m$-cycle such that $\nu_2 \in [\pi_0^m]_{\rho_1}$. Define the following quantities:
\begin{align*}
\tilde{t}_0 \coloneqq t+t^{(c^*+\delta)/c} - t^{(c_{m-1}(\pi^m_0)+\delta)/c}, \text{ and } \\
\tilde{t}_1 \coloneqq t+t^{(c^*+\delta)/c} - t^{(c_{m-1}(\pi^m_0)+\delta/2)/c}.
\end{align*}
Define the stopping time $\theta \coloneqq \inf \{t>\tilde{t}_0: \bar{\mu}(t) \in [\pi^m_0]_{\rho_1}\}$, if $c^* > c_{m-1}(\pi^m_0)$ and $\theta = t_2$ otherwise. By the strong Markov property,
\begin{align}
P_{t_2,\nu_2}& (\bar{\mu}(t+t^{(c^*+\delta)/c}) \in [\tilde{L}_0]_{\rho_1}) \nonumber \\
& \geq E_{t_2,\nu_2}(\theta \leq \tilde{t}_1; E_{\theta, \bar{\mu}(\theta)}(\bar{\mu}(t+t^{(c^*+\delta)/c}) \in [\tilde{L}_0]_{\rho_1} )\nonumber \\
& \geq P_{t_2, \nu_2} (\theta \leq \tilde{t}_1) \inf_{\tilde{t}_0 \leq t_3 \leq \tilde{t}_1, \nu_3 \in [\pi_0^m]_{\rho_1}} P_{t_3,\nu_3}(\bar{\mu}(t+t^{(c^*+\delta)/c}) \in [\tilde{L}_0]_{\rho_1}).
\label{eqn:temp9}
\end{align}
We first estimate $P_{t_2, \nu_2}(\theta \leq \tilde{t}_1)$ when $c^* > c_{m-1}(\pi_0^m)$ (if this is not the case, then by definition of $\theta$, we have $P_{t_2, \nu_2}(\theta \leq \tilde{t}_1)=1$) . Note that
\begin{align*}
P_{t_2,\nu_2} (\theta > \tilde{t}_1)  & = P_{t_2,\nu_2} (\bar{\mu}(t) \notin [\pi^m_0]_{\rho_1} \text{ for all } \tilde{t}_0 \leq t \leq \tilde{t}_1) \\
& \leq P_{t_2,\nu_2} (\bar{\mu}(t) \notin [L]_{\rho_1} \text{ for all } \tilde{t}_0 \leq t \leq \tilde{t}_1)  + P_{t_2, \nu_2}(\bar{\tau}_{\pi^m_0} \leq \tilde{t}_1).
\end{align*}
Lemma~\ref{lemma:hs23} implies that
\begin{align*}
 P_{t_2, \nu_2}(\bar{\tau}_{\pi^m_0} \leq \tilde{t}_1) \leq t^{-\delta/c}
\end{align*}
for large $t$ and small enough $\rho_1 > 0$. Also, with this $\rho_1$, by using Lemma~\ref{lemma:conv-fw19}, we see that
\begin{align*}
P_{t_2,\nu_2} (\bar{\mu}(t) \notin [L]_{\rho_1} \text{ for all } \tilde{t}_0 \leq t \leq \tilde{t}_1)  \leq t^{-M_1/c}
\end{align*}
for large $t$, where $M_1$ can be chosen as large as we want. This shows that there exists $\varepsilon_1>0$ such that
\begin{align*}
P_{t_2,\nu_2} (\theta \leq \tilde{t}_1) \geq 1-2t^{-\varepsilon_1/c}
\end{align*}
uniformly for all $\nu_2 \in [\pi_0^m]_{\rho_1}$ and large enough $t$. Hence, from~(\ref{eqn:temp13}), (\ref{eqn:temp9}) and~(\ref{eqn:temp6}), we get
\begin{align*}
P_{t_1,\nu_1}&(\bar{\mu}(t+t^{(c^*+\delta)/c}) \in [\tilde{L}_0]_{\rho_1}) \\
& \geq (1-t^{-M^\prime/c}) (1-2t^{-\varepsilon_1/c})  \times \inf_{\substack{t_2 \geq \tilde{t}_0, \\ \nu_2 \in [\pi_0^m]_{\rho_1} \\  \pi_0^m \in A_m \\  \tilde{\delta} \in [\delta/4, \delta]}} P_{t_2,\nu_2}(\bar{\mu}(t_2+t_2^{(c_{m-1}(\pi^m_0)+\tilde{\delta})/c}) \in [\tilde{L}_0]_{\rho_1})
\end{align*}
and therefore, for some $\varepsilon>0$, we have
\begin{align*}
\inf_{\substack{t \leq t_1  \leq t+T^*, \\ \nu_1 \in [L]_{\rho_1}}}&P_{t_1,\nu_1}(\bar{\mu}(t+t^{(c^*+\delta)/c}) \in [\tilde{L}_0]_{\rho_1}) \\
& \geq (1-t^{-\varepsilon/c}) \times \inf_{\substack{t_2 \geq \tilde{t}_0 \\ \nu_2 \in [\pi^m_0]_{\rho_1} \\ \pi_0^m \in A_m \\ \tilde{\delta} \in [\delta/4, \delta]}} P_{t_2,\nu_2}(\bar{\mu}(t_2+t_2^{(c_{m-1}(\pi^m_0)+\tilde{\delta})/c}) \in [\tilde{L}_0]_{\rho_1}).
\end{align*}
We now focus on the second term. This probability inside the infimum can be lower bounded using similar steps above starting with~(\ref{eqn:temp9}); instead of the random variable $\theta$, we consider the hitting time of a suitable $(m-1)$-cycle.  Continuing this procedure $m$ times, we eventually reach $A_0$. Therefore, we can show
\begin{align*}
\inf_{\substack{t \leq t_1  \leq t+T^* \\ \nu_1 \in [\tilde{L}_0]_{\rho_1}}}&P_{t_1,\nu_1}(\bar{\mu}(t+t^{(c^*+\delta)/c}) \in [\tilde{L}_0]_{\rho_1})  \geq (1-t^{-\varepsilon/c})^{m+1},
\end{align*}
and the result now follows from~(\ref{eqn:temp7}).
\end{proof}

We now show that the conclusion of Theorem~\ref{thm:conv-globalmin} fails if we choose $c < c^*$. Given $c < c^*$, let $\pi^k \in L_k$ be such that $\hat{V}(\pi^k) \leq c < \tilde{V}(\pi^k)$; this is possible from the definition of $c^*$. Note that $\tilde{L}_0 \cap \pi^k = \emptyset$. The below result shows that the exit time from a neighbourhood of $\pi^k$ is infinite with positive probability, and this in particular implies that~(\ref{eqn:conv_globalmin}) fails.
\begin{proposition}
Let $\pi^k$ be a $k$-cycle such that $\hat{V}(\pi^k) \leq c < \tilde{V}(\pi^k)$. There exist $\varepsilon \in (0, \tilde{V}(\pi^k)-c)$, $c^\prime >0, \rho_1>0$ and $t^*>0$ such that for all $\nu \in [\pi^k]_{\rho_1}\cap M_1^{N_{t}}(\mathcal{Z})$ and $t \geq t^*$, we have
\begin{align*}
P_{t, \nu}(\bar{\tau}_{\pi^k} <\infty) \leq c^\prime t^{1-(\tilde{V}(\pi^k) - \varepsilon)/c}.
\end{align*}
\end{proposition}
\begin{proof}
We proceed via the steps in Hwang and Sheu~\cite{hwang-sheu-90}. Let $T_0 = t$, and define, for all $n \geq 1$,
\begin{align*}
T_{n+1} \coloneqq T_n + T_n^{\hat{V}(\pi^k)/c}, \text{ and}\\
T^*_{n+1} \coloneqq T_n + \frac{1}{2} T_n^{\hat{V}(\pi^k)/c}.
\end{align*}
(In the above definitions, we assume that $\hat{V}(\pi^k) > 0$; if this is not the case, then we replace $T_n^{\hat{V}(\pi^k)/c}$ in the above definitions by a sufficiently large constant, and the following arguments will go through.)
We have, for any $r \geq 1$,
\begin{equation}
P_{t, \nu} (\bar{\tau}_{\pi^k} <T_r) = P_{t, \nu} (\bar{\tau}_{\pi^k} <T_{r-1}) + P_{t, \nu} (T_{r-1} \leq \bar{\tau}_{\pi^k} <T_r).
\label{eqn:temp3}
\end{equation}
To bound the second term, define the stopping time $\theta \coloneqq \inf\{t>T^*_{r-1}:\bar{\mu}(t) \in [L]_{\rho_1}\}$ where $\rho_1$ is to be chosen later. Then,
\begin{align}
P_{t, \nu} (T_{r-1} \leq \bar{\tau}_{\pi^k} <T_r) &  = P_{t, \nu} (T_{r-1} \leq \bar{\tau}_{\pi^k}  <T_r, \theta \leq T_{r-1}^*+T^* )  \nonumber \\
&\, \, \, \, \, \, \, \, \, \, + P_{t, \nu} (T_{r-1} \leq \bar{\tau}_{\pi^k} <T_r, \theta > T_{r-1}^*+T^*) ,
\label{eqn:temp4}
\end{align}
where $T^*$ is such that the second term above is upper bounded by $T_{r-1}^{*^{-M/c}}$ for some $M>0$ to be chosen later (this is possible by Lemma~\ref{lemma:conv-fw19}). To bound the first term, note that
\begin{align*}
P_{t, \nu} (T_{r-1} \leq &\bar{\tau}_{\pi^k} <T_r,  \theta \leq T_{r-1}^*+T^* )  \\
& \leq P_{t, \nu} (\theta\leq \bar{\tau}_{\pi^k} <T_r, \theta \leq T_{r-1}^*+T^* ) \\
& \leq E_{t,\nu} (\bar{\mu}(\theta) \in [\pi^k]_{\rho_1}, \theta \leq  T_{r-1}^*+T^* ; E_{\theta, \bar{\mu}(\theta)}(\bar{\tau}_{\pi^k} < T_r)) \\
& \leq T_{r-1}^{*-(\tilde{V}(\pi^k) - \hat{V}(\pi^k) - \varepsilon)/c}
\end{align*}
holds for sufficiently large $t$ and small enough $\rho_1$. Here, the second inequality follows by the strong Markov property and the third from Lemma~\ref{lemma:hs24}. Choose $M$ sufficiently large, so that~(\ref{eqn:temp3}),~(\ref{eqn:temp4}) and the above implies
\begin{align*}
P_{t,\nu}(\bar{\tau}_{\pi^k} < T_r) \leq P_{t,\nu}(\bar{\tau}_{\pi^k} < T_{r-1}) + 2T_{r-1}^{*^{-(\tilde{V}(\pi^k) - \hat{V}(\pi^k) -\varepsilon)/c}}.
\end{align*}
Therefore, we have
\begin{align*}
P_{t,\nu}(\bar{\tau}_{\pi^k} < T_r)  &\leq 2\sum_{n=0}^r T_n^{*^{-(\tilde{V}(\pi^k) - \hat{V}(\pi^k) -\varepsilon)/c}} \\
& \leq c_1^\prime  \sum_{n=0}^r T_n^{-(\tilde{V}(\pi^k) - \hat{V}(\pi^k) -\varepsilon)/c} \\
& = c_1^\prime \sum_{n=0}^r T_n^{-(\tilde{V}(\pi^k) - \varepsilon)/c} (T_{n+1} - T_n)\\
& \leq c_1^\prime \int_{t}^{T_r} u^{-(\tilde{V}(\pi^k)-\varepsilon)/c} du,
\end{align*}
where $c_1^\prime $ is a positive constant. Choose $\varepsilon $ such that $\tilde{V}(\pi^k) - \varepsilon > c$ so that the above implies
\begin{align*}
P_{t,\nu}(\bar{\tau}_{\pi^k} < T_r) &\leq c_1^\prime \int_{t}^ \infty u^{-(\tilde{V}(\pi^k)-\varepsilon)/c} du \\
& \leq c^\prime t^{1-(\tilde{V}(\pi^k)-\varepsilon)/c},
\end{align*}
where $c^\prime$ is a positive constant. Let $r \to \infty$, and the result follows since $T_r \to  \infty$.
\end{proof}

\section*{Acknowledgements}
The authors would like to thank Laurent Miclo for fruitful discussions and Siva Athreya for suggestions on the organisation of the paper.
\bibliographystyle{abbrv}
\bibliography{Report}
\end{document}